\documentclass[10pt]{amsart}
 \usepackage[utf8]{inputenc}
 \usepackage[margin=1.3in]{geometry}
 \usepackage{url,float,tikz,graphicx,tikz-cd}
 \usepackage{amsmath,amsthm,amsfonts,amssymb,amscd,latexsym,bm,enumitem,mathrsfs,mathtools}
\usepackage[english]{babel}
 \usepackage{csquotes}

\usepackage{eepic,epsfig,eufrak}
\usepackage[english]{babel}
\usepackage{hyperref}

\theoremstyle{plain}
\newtheorem{theorem}[equation]{Theorem}
\newtheorem{lemma}[equation]{Lemma}
\newtheorem{proposition}[equation]{Proposition}
\newtheorem{proposition-definition}[equation]{Proposition-Definition}

\theoremstyle{definition}
\newtheorem{definition}[equation]{Definition}
\newtheorem{example}[equation]{Example}
\newtheorem{remark}[equation]{Remark}
\newtheorem{block}[equation]{}

\newcommand{\br}{\mathrm{Br}}
\newcommand{\an}{\mathrm{an}}

\renewcommand{\log}{\mathrm{log}}
\newcommand{\ord}{\mathrm{ord}}

\newcommand{\Hom}{\mathrm{Hom}}

\newcommand{\add}[1]{\add \emph{#1}}

\newcommand{\val}{\mathrm{val}}
\renewcommand{\div}{\mathrm{div}\,}

\newcommand{\red}{\mathrm{red}}

\newcommand{\Gal}{\mathrm{Gal}}
\newcommand{\wt}{\mathrm{wt}}
\renewcommand{\sp}{\mathrm{sp}}

\newcommand{\coker}{\mathrm{coker}}

\newcommand{\sk}{\mathrm{sk}}

\renewcommand{\br}{\mathrm{Br}}
\renewcommand{\red}{\mathrm{red}}
\newcommand{\lra}{\longrightarrow}

\newcommand{\colim}{\mathrm{colim}}

\renewcommand{\P}{\mathbb{P}}
\newcommand{\Q}{\mathbb{Q}}

\newcommand{\R}{\mathbb{R}}

\newcommand{\Z}{\mathbb{Z}}

\renewcommand{\O}{\mathscr{O}}
\renewcommand{\H}{\mathscr{H}}

\newcommand{\CC}{\mathscr{C}}

\newcommand{\M}{\mathscr{M}}
\newcommand{\YY}{\mathscr{Y}}

\newcommand{\ZZ}{\mathscr{Z}}

\newcommand{\leaf}{\mathrm{leaf}}

\newcommand{\mult}{\mathrm{mult}}

\renewcommand{\add}[1]{$\bigstar\bigstar\bigstar\bigstar\bigstar$\emph{\footnotesize{#1}}}
\numberwithin{equation}{section} 

\begin{document}
\title[The different for base change of arithmetic curves]{the different for base change of arithmetic curves}

\date{\today}
\author[Art Waeterschoot]{Art Waeterschoot}
\address{KU Leuven\\
Department of Mathematics\\
Celestijnenlaan 200B\\3001 Heverlee \\
Belgium}
\email{art.waeterschoot@kuleuven.be}
\urladdr{\url{https://sites.google.com/view/artwaeterschoot}}
\thanks{The author was supported by the Fund for Scientific Research Flanders (PhD Fellowship 11F0123N and Research Project G0B1721N). This work is part of the author's PhD thesis supervised by Johannes Nicaise, who suggested to look into \ref{thm: ell curve}. I thank him and also Stefan Wewers, Lars Halvard Halle, Dennis Eriksson, Michael Temkin, Micha\"el Maex for stimulating discussions.}
\begin{abstract} We introduce a method for studying reduction types of arithmetic curves and wildly ramified base change. We give new proofs of earlier results of Lorenzini \cite{L10} and Obus-Wewers \cite{OW}, and resolve a question of Lorenzini \cite{L14} on the Euler characteristic of the resolution graph of a $p$-cyclic arithmetic surface quotient singularity. Our method consists of constructing a simultaneous skeleton for the associated cover of Berkovich analytifications and applying the skeletal Riemann-Hurwitz formula of \cite{W}. \end{abstract}

\maketitle

\tableofcontents
\addtocontents{toc}{\protect\setcounter{tocdepth}{1}}

\setcounter{tocdepth}{1}
\section{Introduction}
\label{sec: intro}
\begin{block}[Aim]
The objective of this paper is to complement the work \cite{W} with applications and more explicit results in the case of curves. The main new result is Theorem \ref{thm: combinatorial slope}. 
\end{block}
\begin{block}[Notation] We fix the following notation.\label{notation}
		\begin{enumerate}[label=(\roman*)]
			\item Let $k$ denote a henselian discretely valued field with normalised valuation $v_k$, ring of integers $k^{\circ}$ and algebraically closed ground field $\tilde{k}$ of characteristic $p$. 
			\item Let $C$ denote a smooth projective geometrically connected algebraic $k$-curve.
			\item A \emph{$k^{\circ}$-model of $C$} is a proper normal flat $k^{\circ}$-scheme $\CC$ equipped with a generic fiber isomorphism $\CC_\eta\cong C$.  A $k^{\circ}$-model $\CC$ is equipped with the ``standard'' logarithmic structure $\CC^{\dagger}\coloneqq (\CC,\M_{\CC})$ where $\M_{\CC}$ is the sheaf of functions invertible on the generic fiber. We call a $k^{\circ}$-model $\CC$ \emph{toroidal} if the log-scheme $\CC^{\dagger}$ is logarithmically regular \cite{K94}. By op.cit. \S11.6 regular toroidal $k^{\circ}$-models coincide with \emph{snc} $k^{\circ}$-models, defined as regular $k^{\circ}$-models $\CC$ for which the reduced special fiber $\CC_{s,\red}$ is an \emph{snc} divisor. We prefer to work with the larger class of toroidal models because these also admit a theory of skeleta and minimal models and they behave better with respect to base change (see Section \ref{sec: models}). 
			\item Given a toroidal $k^{\circ}$-model $\CC$ we write $\Gamma(\CC)$ for the graph whose vertices are the generic points of $\CC_s$ and edges correspond to non-empty intersections. To be precise, for a vertex $x\in V(\Gamma(\CC))$ we write $C_x\coloneqq\overline{\{x\}}$ for the corresponding component of $\CC_{s}$ and we draw an edge from $x$ to $x'$ if and only if $C_x\cap C_{x'}\ne \emptyset$. The sheaf $\M_{\CC}^{\sharp}$ of toric coordinates on $\CC$ equips $\Gamma$ with a canonical metric (see \ref{definition skeleton}) called the \emph{potential} metric \cite[\S2]{BN} or \emph{conformal} metric \cite[\S4]{W}, beware it is not stable under base change.
		\end{enumerate}
\end{block}
\begin{block}[Skeleta] For every toroidal $k^{\circ}$-model there exists a canonical topological embedding $\Gamma(\CC) \hookrightarrow C^{\an}$ where $C^{\an}$ is the Berkovich $k$-analytification of $C$. In fact $\Gamma(\CC)$ is a strong deformation retract of $C^{\an}$, therefore any subspace $\Gamma$ of the form $\Gamma(\CC)$ will be called a \emph{skeleton}. The \emph{tropical canonical divisor} $K_{\Gamma}$ of $\Gamma$ is defined as a formal linear combination of the vertices of $\Gamma$ with coefficient at $x\in V(\Gamma)$ equal to $$m(x)\left(\chi(x)-\mathrm{val}_{\Gamma}(x)\right);$$ here
\begin{enumerate}[label=(\roman*)]
 	\item $\chi(x)\coloneqq \chi(C_x)=2-2g(C_x)$ is the Euler characteristic of $C_x$.
 	\item $\mathrm{val}_\Gamma(x)$ is the valency of $x$ in $\Gamma$, i.e. the number of neighbouring components of $C_x$
 	\item $m(x)=e(\O_{\CC,x}/k^{\circ})=\mult_{C_x}\CC_s$ denotes the multiplicity of $C_x$.
 \end{enumerate} The definition of $K_{\Gamma}$ is equally valid for any \emph{subgraph} (assumed compact connected, see Definition \ref{subgraph}) $\Gamma$ of $C^{\an}$  and one defines $$\chi(\Gamma)\coloneqq\deg K_{\Gamma}.$$
\end{block}
\begin{theorem}[Enlarging skeleta]
 \label{thm: enlarging skeleta} Suppose that $\Gamma$ is a subgraph of $C^{\an}$ which contains a skeleton. Then $\Gamma$ is a skeleton if and only if $\chi(\Gamma)=\chi(C)$.
\end{theorem}
\begin{remark} For instance, the interval $I=\{\eta_{0,r}:r\in[0,1/2]\}$ of $(\P^1)^{\an}$ is not a skeleton because $\chi(I)=3\ne 2$.
		Theorem \ref{thm: enlarging skeleta} contrasts with the situation over algebraically closed fields, where the condition on $\chi$ is not required, see \cite[\S3.5.4]{CTT}. 
\end{remark}
\begin{block}[Simultaneous skeleta and Riemann-Hurwitz]\label{intro diff}
Now let $C'=C\otimes_kk'$ denote the base change of $C$ to a finite seperable extension $k'/k$. Denote by $\pi:C'\to C$ the projection. Let $\CC$ be a toroidal $k^{\circ}$-model of $C$ and let $\CC'$ denote the normalisation of $\CC\otimes_{k^{\circ}}(k')^{\circ}$. Write $\Gamma'=(\pi^{\an})^{-1}(\Gamma)$. Assume that the conclusion of \ref{thm: enlarging skeleta} holds for $\Gamma'$, then $\CC'$ is a toroidal $(k')^{\circ}$-model with skeleton $\Gamma'$. We will also call $\Gamma'\to\Gamma$ a simultaneous skeleton of $C'^{\an}\to C^{\an}$. The following two results where shown in \cite{W}, see Section \ref{sec: bc} for details:
\begin{enumerate}[label=(\roman*)]
	\item 
First, $\Gamma'\to\Gamma$ is a \emph{harmonic} cover of metric graphs, meaning that for each edge $e$ of $\Gamma$ we have $$[k':k]=\sum_{e'/e} \mathrm{length}(e)/\mathrm{length}(e')$$ summing over edges $e'$ of $\Gamma'$ above $e$. 
\item Secondly let $\delta:\Gamma'\to\R_{\ge0}$ denote the \emph{different function}; it can be defined as the unique continuous function which is linear on the edges of $\Gamma'$ such that at a vertex $x'\in V(\Gamma')$ we have $$\delta(x')
=\frac{1}{m(x')}\mathrm{length}_{\O_{\CC',x'}}\delta^{\log}_{\O_{\CC',x'}/\O_{\CC,x}}$$ where $\delta^{\log}$ is the additive log-different (see \ref{sec: diff}). The combinatorial Laplacian $\Delta(\delta)$ is the combinatorial divisor whose coefficients are the sum of outgoing slopes. Then we have the Riemann-Hurwitz formula
$$\Delta(\delta)=K_{\Gamma'}-\pi^*K_{\Gamma}.$$
\end{enumerate}
\end{block}
\begin{block}[Complements on the different function]
Additionally, we have the following properties of the different function $\delta:\Gamma'\to\R_{\ge0}$:
\begin{enumerate}[label=(\roman*)]
	\item For all $x'\in V(\CC')$ we have $\delta(x')=0$ if and only if $\O_{\CC',x'}/\O_{\CC,x}$ is tamely ramified (Proposition \ref{propn: (log) differents} (e)).
	\item Let $\Gamma^{t}$ denote the \emph{temperate} part of $\Gamma$, this is the closure of the tame divisorial points on $\Gamma$ \cite{JN}. Then $\delta$ is constant on $(\Gamma^t)'$, with value equal to $\frac{1}{[k':k]}\delta^{\log}_{k'/k}$ (Proposition \ref{prop: temp})
	\item Suppose $\kappa(x')/\kappa(x)$ is seperable, then the slope $\partial_{b'}\delta$ of $\delta$ at along a branch $b'$ at $x'$ is equal to $m(x')\delta^{\log}_{\O_{C_{x'},b'}/\O_{C_{x},b}}$ (Proposition \ref{comparison RH}).
\item The slopes of $\delta$ are integers
\item $\delta$ is bounded by $v_k[k':k]$ (see \ref{bounds different}). 
\end{enumerate}
Only properties (i)-(iii) will be used in the applications.
\end{block}
\begin{block}[Application: the 'base change' method]
In the applications we obtain information on the dual graph $\Gamma$ in some situations. The idea of the method is to apply the properties of the different for a base change $\Gamma'\to\Gamma$. The first application is a computation-free proof of the following Theorem of Lorenzini \cite[2.8]{L10}.
\end{block}
\begin{theorem}[Potentially multiplicative elliptic curves] \label{thm: ell curve} Let $E/k$ be an elliptic curve with bad reduction and potentially multiplicative reduction. Write $\nu=-\ord_kj(E)$. Then $E$ has multiplicative reduction over an extension $k'/k$ of degree $2$ and $E$ has Kodaira-N\'eron reduction type $I^*_{\nu+4\delta^{\log}_{k'/k}}$ and $E_{k'}$ has Kodaira-N\'eron reduction type $I_{2\nu}$. 
\end{theorem}

\begin{example}\label{example pot mult}
Let us give a concrete example supporting Theorem \ref{thm: ell curve}, illustrating the connection with the different function.
	Let $k=\Q_2$, $k'=\Q_2(\sqrt{5})$ and let $E$ be the elliptic curve with Legendre equation
$y^2=x(x-1)(x-2^5).$ We have $\nu=-\mathrm{ord}_kj(E)=2$ and $\delta^{\log}_{k'/k}=1$, and a computation via Tate's algorithm (implemented in SageMath) confirms that $E$ and $E_{k'}$ have Kodaira-N\'eron reduction type $I_5^*$ and $I_2$ respectively. Let $\Gamma$ be the minimal skeleton of $E$, then the proof of \ref{thm: ell curve} will show that $\Gamma'\to\Gamma$ is a simultaneous skeleton and the behaviour of $\delta:\Gamma'\to \R$ is depicted in Figure \ref{figure example pot mult} below. Also see Example \ref{example II} for a more involved example with potentially supersingular reduction.
\end{example}
\refstepcounter{equation}
 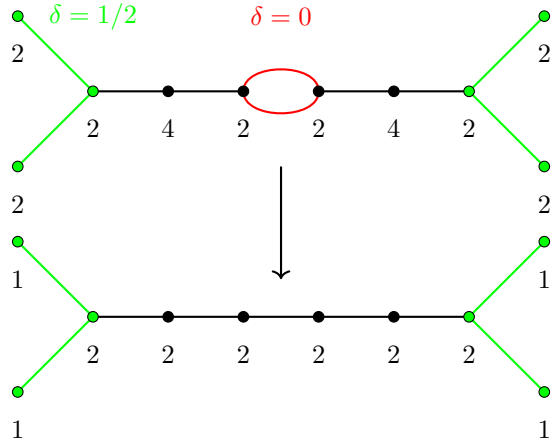
\begin{figure}[ht!]
\centering
	\begin{tikzpicture}
 \label{figure example pot mult}

\draw[thick,->] (0.5,2)--(0.5,0.5);

\node[green] at (-2,4) {$\delta=1/2$};
\node[red] at (0.5,4) {$\delta=0$};

\node at (-3,3.5) {$2$};
\node at (-3,1.5) {$2$};
\node at (-2,2.5) {$2$};
\node at (-1,2.5) {$4$};
\node at (0,2.5) {$2$};
\node at (1,2.5) {$2$};
\node at (2,2.5) {$4$};
\node at (3,2.5) {$2$};
\node at (4,1.5) {$2$};
\node at (4,3.5) {$2$};

\draw[thick] (-2,3) -- (0,3);
\draw[thick] (1,3) -- (3,3);
 \draw[red,thick] (0,3) to [out=90,in=90] (1,3);
 \draw[red,thick] (1,3) to [in=-90,out=-90] (0,3); 
\draw[green,thick] (4,4) -- (3,3)-- (4,2);
\draw[green, thick] (-3,4) -- (-2,3) --(-3,2);

\node at (-3,0.5) {$1$};
\node at (-3,-1.5) {$1$};
\node at (-2,-0.5) {$2$};
\node at (-1,-0.5) {$2$};
\node at (0,-0.5) {$2$};
\node at (1,-0.5) {$2$};
\node at (2,-0.5) {$2$};
\node at (3,-0.5) {$2$};
\node at (4,-1.5) {$1$};
\node at (4,0.5) {$1$};
\draw[thick] (-2,0) -- (3,0);
\draw[green,thick] (4,1) -- (3,0)-- (4,-1);
\draw[green, thick] (-3,1) -- (-2,0) --(-3,-1);

\draw[fill=green] (-3,4) circle (2pt);
\draw[fill=green] (-3,2) circle (2pt);
\draw[fill=green] (-2,3) circle (2pt);
\draw[fill=black] (-1,3) circle (2pt);
\draw[fill=black] (0,3) circle (2pt);
\draw[fill=black] (1,3) circle (2pt);
\draw[fill=black] (2,3) circle (2pt);
\draw[fill=green] (3,3) circle (2pt);
\draw[fill=green] (4,4) circle (2pt);
\draw[fill=green] (4,2) circle (2pt);

\draw[fill=green] (-3,1) circle (2pt);
\draw[fill=green] (-3,-1) circle (2pt);
\draw[fill=green] (-2,0) circle (2pt);
\draw[fill=black] (-1,0) circle (2pt);
\draw[fill=black] (0,0) circle (2pt);
\draw[fill=black] (1,0) circle (2pt);
\draw[fill=black] (2,0) circle (2pt);
\draw[fill=green] (3,0) circle (2pt);
\draw[fill=green] (4,1) circle (2pt);
\draw[fill=green] (4,-1) circle (2pt);

	\end{tikzpicture}
\caption{\small \emph{Picture of a simultaneous skeleton $\Gamma'\to\Gamma$ for Example \ref{example pot mult}. The markings are the multiplicities of the vertices. The different $\delta$ is constant on the coloured part of $\Gamma$ and varies linearly with slope $2$ from {\color{red}$\delta=0$} on the loop} to {\color{green}$\delta=1/2$} above $\Gamma^{\mathrm{t}}$}.
\end{figure}\label{supporting example}
\begin{block} Our second application concerns wild quotient singularities. Suppose $k'/k$ is a $p$-cyclic extension and assume $C'=C_{k'}$ admits a smooth $k^{\circ}$-model $\YY/k'^{\circ}$. That is, we assume that $C$ admits \emph{good reduction over $k'$}. Let $\ZZ=\YY/\mathrm{Gal}(k'/k)$, this is a singular $k^{\circ}$-model of $C$. The singular points of $\ZZ$ correspond to the branching points of the cover $\YY_s\to \ZZ_s$. Let $Q$ be a branch point, and let $P$ be the (unique) ramification point above $Q$. Let $\CC\to \ZZ$ be a resolution of $\ZZ$. Write $\Gamma_{Q}\in\pi_0(\Gamma\setminus\{x\})$ for the connected component of $\Gamma\setminus \{x\}$ corresponding to $Q$.\end{block}
\begin{theorem}\label{thm: combinatorial slope} Let $j_Q$ be the unique lower ramification jump of $\O_{\YY_s,P}/\O_{\ZZ_s,Q}$, then $$\chi(\Gamma_{Q})=1-(p-1)j_Q.$$
\end{theorem}
\begin{remark} Theorem \ref{thm: combinatorial slope} shows that the ramification theory of $\O_{P}/\O_{Q}$ adequately measures the ``combinatorial'' complexity of the quotient singularity at $Q$, and answers
	\label{rmk: comb slope} the main question which was left open in \cite{L14}: A translation of notation shows that the invariants $v(Q)$ and $\gamma g$ of Section 6.2 in op. cit. coincide with $\delta_{\O_{P}/\O_{Q}}$ and $p-\chi(\Gamma_{Q})$ respectively.
\end{remark}
\begin{block}[Weakly wild case] More information on the singularities can be obtained by imposing stronger conditions. In the setting of \ref{thm: combinatorial slope} we call the singularity of $\ZZ$ at $Q$ \emph{weakly wild} if $j_Q=1$; such singularities were studied in detail in \cite{L14,OW}, in particular Lorenzini \cite{L14} proves that $\Gamma_Q$ is given by Figure \ref{graph} below. Lorenzini \cite[p.333, guess of $\alpha$]{L14} asked if the number of multiplicity $p$ components on the horizontal segment of $\Gamma_Q$ equals $jp$, where $j$ is the unique lower ramification jump of $k'/k$. This was recently affirmed by Obus-Wewers \cite[Corollary 7.13]{OW} using methods from deformation theory. In Section \ref{sec:quotient} we give another proof in the ordinary case via a slope computation of the different function for the base change $\Gamma'_Q\to \Gamma_Q$. 
\end{block}
\refstepcounter{equation}
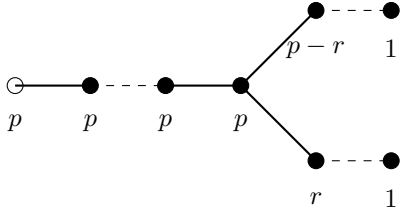
\begin{figure}[ht!]
\label{graph}
\centering
	\begin{tikzpicture}
		\draw (0,0) circle (3pt);
\draw[fill=black] (1,0) circle (3pt);
\draw[fill=black] (2,0) circle (3pt);
\draw[fill=black] (3,0) circle (3pt);
\draw[fill=black] (4,1) circle (3pt);
\draw[fill=black] (4,-1) circle (3pt);
\draw[fill=black] (5,1) circle (3pt);
\draw[fill=black] (5,-1) circle (3pt);

\node at (0,-0.5) {$p$};
\node at (1,-0.5) {$p$};
\node at (2,-0.5) {$p$};
\node at (3,-0.5) {$p$};
\node at (4,-1.5) {$r$};
\node at (5,-1.5) {$1$};
\node at (4,0.5) {$p-r$};
\node at (5,0.5) {$1$};
\draw[thick] (0,0) -- (1,0);
\draw[thick] (2,0) -- (3,0)--(4,1);
\draw[thick] (3,0) -- (4,-1);
\draw[dashed] (1,0) -- (2,0);
\draw[dashed] (4,-1) -- (5,-1);
\draw[dashed] (4,1) -- (5,1);
	\end{tikzpicture}
	\caption{Resolution graph $\Gamma_Q$ of a weakly wild $p$-cyclic arithmetic surface quotient singularities; we have $j_Q=1$ and $\chi(\Gamma_Q)=2-p$. The markings indicate multiplicities and $r$ is an integer such that $0<r<p$.}
\end{figure}
\begin{block}[Related work] Temkin and collaborators \cite{T16,CTT,T17,BT} made a complete study of the different function over algebraically closed fields. 
The different function goes back to work of Gabber and L\"utkebohmert, see the references cited in \cite{CTT}. 

Some relevant works on $2$-dimensional quotient singularities include \cite{brieskorn,artin,L13,IS,OW}, and regarding wild covers and wild base change of arithmetic curves we mention \cite{Saito,LL,L10,N12,L14,HN,OT,FT,DM23,ole}. For work related to skeleta, harmonic covers and the Riemann-Hurwitz formula see the references cited in \cite{W}.
\end{block}

\begin{remark}[Outlook] Some natural questions are not pursued in this paper and will be investigated in future work. First, it is natural to ask how to interprete Theorem \ref{thm: combinatorial slope} as a McKay correspondence in the sense of \cite{Y}. 
Second, we expect that most properties of the different function shown in \cite{CTT,T17} extend to the discretely valued setting, except for radiality of multiplicity loci. Nevertheless one can wonder if multiplicity loci are definable in a suitable first-order language like in the algebraically closed setting \cite{CKP}.
\end{remark}

\begin{block}[Overview]
In Section \ref{sec: diff} we give background on the different of finite extensions of discretely valued fields. 
In Section \ref{sec: models} we discuss models of curves and skeleta, and prove Theorem \ref{thm: enlarging skeleta}. In section \ref{sec: bc} we discuss properties of the different function associated to base change of skeleta. In Section \ref{sec:ell curve} give examples and a new proof of Theorem \ref{thm: ell curve}. In Section \ref{sec:quotient} we prove Theorem \ref{thm: combinatorial slope}. 
\end{block}

\section{Preliminaries on differents of discretely valued fields}
 \label{sec: diff}
 \begin{block} In this section we give some background on the different ideal for extensions of DVRs, see Proposition \ref{propn: (log) differents} for a summary.
\end{block}
\begin{block}[notation for this section]
 Throughout this section $E/F$ denotes a finite seperable extension of discretely valued fields and $E^{\circ}/F^{\circ}$ denotes the associated extension of rings of integers. The seperability assumption implies that the trace morphism $E\mapsto \mathrm{Hom}_{F}(E,F):1\mapsto \mathrm{Tr}_{E/F}(\cdot)$ is an isomorphism, and it is a well-known corollary that $E^{\circ}$ is a finite free $F^{\circ}$-module, see \cite[032L]{stacks}. In particular it follows that $\Omega_{E^{\circ}/F^{\circ}}$ is a finite sum of cyclic modules.
 \end{block}
\begin{definition}  \label{definition different} 
The \emph{(additive) different} of $E/F$ is defined as the integer $$\delta_{E^{\circ}/F^{\circ}} \coloneqq \mathrm{length}_{E^{\circ}}\Omega_{E^{\circ}/F^{\circ}}.$$ The \emph{(additive) log-different} is defined as $$\delta^{\log}_{E/F}=\mathrm{length}_{E^{\circ}}\Omega^{\log}_{E^\circ/F^\circ},$$ here $E^{\circ}$ and $F^{\circ}$ are equipped with the \emph{standard} log-structures defined by $E^{\circ}\setminus \{0\}\to E^{\circ}$ and $F^{\circ}\setminus \{0\}\to F^\circ$ and $\Omega^{\log}_{E^{\circ}/F^{\circ}}$ is finitely generated by \cite[IV.1.2]{ogus}. We will omit the adjective ``additive'' in the sequel. We will sometimes write $\delta_{E^{\circ}/F^{\circ}}$ instead to $\delta_{E/F}$ and similarly for the logarithmic variant.
\end{definition}

\begin{block}[reminder on determinants and fitting ideals] 
	\label{fitting ideals} If $M$ is a finitely generated $E^\circ$-module then the determinant $\det M$ as in \cite{KM} is a natural $E^{\circ}$-submodule of $\det M_{E}=\wedge^{\mathrm{rank}_EM_{E}}M_E$; up to natural isomorphism $\det M$ is defined as the alternating product $\bigotimes \left(\wedge^{\mathrm{rank} P^i}P^i\right)^{(-1)^i}$ of top wedge products of the terms in a finite length projective resolution $P^\bullet \to M$. For each short exact sequence $0\to M'\to M\to M''\to 0$ of finitely generated $E^{\circ}$-modules there is a natural adjunction $\det M\cong \det M'\otimes \det M''$, 
	Now suppose $M$ is torsion, then $\det M$ can be viewed as a fractional ideal and in fact $$\det M\cong \left(\mathrm{Fitt}^0M\right)^{-1}$$ where $\mathrm{Fitt}^0M$ denotes the 0'th Fitting ideal, by unwinding the definition \cite[07Z9]{stacks}. Explicitly, by Smith normal form theory $M$ is a sum of cyclic modules, say $M\cong \bigoplus_{i=1}^nE^{\circ}/(a_i)$ for some $a_1,\dots,a_n\in E^{\circ}$ and we have $(\det M)^{-1}=\mathrm{Fitt}^0(M)=(a)$ where $a$ is the product of the $a_i$. Note that $v_E(a)=\mathrm{length}\,M$. 
\end{block} 
\begin{block}[relation with different ideals and canonical modules] By \ref{fitting ideals} we have the following two alternative descriptions of the (log-)different. \label{2 descriptions diff}
	\begin{enumerate}[label=(\roman*)]
	\item The different is the normalised valuation of a generator of the \emph{different ideal\footnote{\label{rmk: canonical is dualising} In the literature, for instance in \cite{local fields}, the different ideal $\mathscr{D}_{E/F}$ is usually defined as the inverse of the trace dual -- the definitions agree by Proposition \ref{propn: (log) differents} (d) below }} $$\mathscr{D}_{E/F}\coloneqq \mathrm{Fitt}^0\Omega_{E^{\circ}/F^{\circ}},$$ and similarly the log-different $\delta_{E/F}^\log$ is the normalised valuation of a generator of the \emph{log-different ideal} $$\mathscr{D}_{E/F}^{\log}\coloneqq \mathrm{Fitt}^0\Omega_{E^{\circ}/F^{\circ}}^{\log}.$$

\item 	\label{relation with canonical modules} The canonical and log-canonical module of $E^{\circ}/F^{\circ}$ are defined as $$\omega_{E^{\circ}/F^{\circ}}\coloneqq \det \Omega_{E^{\circ}/F^{\circ}}\cong\mathscr{D}_{E/F}^{-1}$$ and $$\omega_{E^{\circ}/F^{\circ}}^{\log}\coloneqq \det \Omega_{E^{\circ}/F^{\circ}}^{\log}\cong\left(\mathscr{D}_{E/F}^{\log}\right)^{-1}.$$ The image of $1\in E$  in the (log-)canonical module of $E^{\circ}/F^{\circ}$ is called the \emph{trace element}. Then the (log-)different equals the order of vanishing of $\tau_{E^{\circ}/F^{\circ}}$ as a section of the invertible module $\omega_{E^{\circ}/F^{\circ}}^{(\log)}$. 
	\end{enumerate}
	As a consequence, when given a complete intersection presentation $E^{\circ}\cong F^{\circ}[x_1,\dots,x_k]/(f_1,\dots,f_k)$, the different can be computed as the valuation of the Jacobian determinant: $$\delta_{E/F}=v_E\left(\det\left(\left(\partial f_i/\partial x_j\right)_{ij}\right)\right).$$  
\end{block}

\begin{block}[Reminder on the Riemann-Hurwitz formula] \label{classical RH}The different is of geometric significance because it plays the role of the local term in the classical Riemann-Hurwitz formula: if $f:C\to D$ is a finite separable morphism of proper smooth curves over a field $k$ then we have a short exact sequence $0\to f^*\Omega_{D/k}\to \Omega_{C/k}\to \Omega_{C/D}\to 0$, applying determinants we obtain an adjunction $g^*\omega_{D/k}\otimes \omega_{C/D}\cong \omega_{C/k}$ and so $\omega_{C/D}\cong \underline{\mathrm{Hom}}(f^*\omega_{D/k},\omega_{C/k})$ and the trace element $\tau_{C/D}\in \Gamma(C,\omega_{C/D})$ is defined as the global section of $\omega_{C/D}$ corresponding to the pullback $\omega\mapsto f^*\omega$ of canonical forms. For each $c\in C$ the trace element $\tau_{C/D}$ localises to the trace element $\tau_{\O_{C,c}/\O_{D,f(c)}}\in \omega_{\O_{C,c}/\O_{D,d}}$. Hence $$\sum_{c\in C}\delta_{\O_{C,c}/\O_{D,f(c)}}=\deg \div\tau_{C/D}=\deg \omega_{C/D}=-\chi(C)+(\deg f)\chi(D).$$ 
\end{block}

\begin{proposition}[Properties of the different] \label{propn: (log) differents}
	Let $E/F$ be a finite seperable extension of discretely valued fields.
	\begin{enumerate}[label=(\alph*)]
	\item We have $$\delta^{\log}_{E/F}=\delta_{E/F}+1-e(E/F)\ge0,$$ where $e(E/F)$ denotes the ramification index of $E/F$.
	\item If $E/F/G$ is a tower of finite seperable extensions of discretely valued fields, then $$\delta_{E/G}=\delta_{E/F}+e(E/F)\delta_{F/G}\quad\text{and}\quad\delta^{\log}_{E/G}=\delta^{\log}_{E/F}+e(E/F)\delta^{\log}_{F/G}.$$
	\item Let $\hat{E}/\hat{F}$ denote the completion of $E/F$. Then $\delta_{\hat{E}/\hat{F}}=\delta_{E/F}$ and $\delta_{\hat{E}/\hat{F}}^{\log}=\delta_{E/F}^{\log}$.

	\item The \emph{trace morphism} $E\to \Hom_F(E,F):e\mapsto \mathrm{Trace}_{E/F}(e\cdot)$ 
restricts to an isomorphism \begin{equation}\omega_{E^{\circ}/F^{\circ}}\cong \Hom_{F^{\circ}}(E^{\circ},F^{\circ})\label{eqn: kcan} 
\end{equation} 
identifying the canonical module with the dualising module
, such that the trace element $\tau_{E^{\circ}/F^{\circ}}$ corresponds to the usual trace map.
	\item The following are equivalent:
	\begin{enumerate}
		\item[(i)] The extension $E^{\circ}/F^{\circ}$ is unramified  (resp. tamely ramified) 
		\item[(ii)] $\delta_{E/F}=0$ (resp. $\delta^{\log}_{E/F}=0$)
		\item[(iii)] $\tau_{E^{\circ}/F^{\circ}}$ generates $\omega_{E^{\circ}/F^{\circ}}$ (resp. $\tau_{E^{\circ}/F^{\circ}}$ generates $\omega^{\log}_{E^{\circ}/F^{\circ}}$).
	\end{enumerate}
	
	\item If $\mathrm{char}\,E=0$ we have $\delta^{\log}_{E/F}\le v_E(e(E/F)f^i(E/F)),$ where $f^i(E/F)$ denotes the inseperability degree of $\widetilde{E}/\widetilde{F}$.
\end{enumerate}
\end{proposition}
\begin{proof} Part (a): by \cite[\S7.28]{W} we have $\omega_{E^{\circ}/F^{\circ}}^{\log}=\pi_E^{e(E/F)-1}\omega_{E^{\circ}/F^{\circ}},$ as fractional ideals (for some choice of uniformiser $\pi_E$ of $E$). By \ref{2 descriptions diff}\ref{relation with canonical modules} this implies $\mathscr{D}^{\log}_{E/F}=\pi_E^{1-e(E/F)}\mathscr{D}_{E/F}$ and the desired equality follows by \ref{2 descriptions diff}(i). See \cite[\S5.4.9]{T16} for another proof (which generalises substantially).

Part (b): \label{adjunction} by \cite[\S7.10]{W} we have exact sequences $0\to \Omega_{F^{\circ}/G^{\circ}}\otimes_{F^{\circ}}\to \Omega_{E^{\circ}/G^{\circ}}\to  \Omega_{E^{\circ}/F^{\circ}}\to 0$ and $0\to \Omega_{F^{\circ}/G^{\circ}}^{\log}\otimes_{F^{\circ}}\to \Omega^{\log}_{E^{\circ}/G^{\circ}}\to  \Omega^{\log}_{E^{\circ}/F^{\circ}}\to 0.$
	Applying determinants, we obtain the \emph{adjunction formulae} $\omega_{F^{\circ}/G^{\circ}}\otimes_{F^{\circ}} \omega_{E^{\circ}/G^{\circ}}\cong \omega_{E^{\circ}/F^{\circ}}$ and $\omega_{F^{\circ}/G^{\circ}}^{\log}\otimes_{F^{\circ}} \omega_{E^{\circ}/G^{\circ}}^{\log}\cong \omega_{E^{\circ}/F^{\circ}}^{\log}$, so the formation of (log-)canonical modules is compatible in towers, and the conclusino follows from \ref{2 descriptions diff}. 

For part (c) note that $\hat{F}^{\circ}\otimes_{F^{\circ}} E^{\circ}$ is a finite $\hat{F}^{\circ}$-module, which is therefore complete and so $$\hat{F}^{\circ}\otimes_{F^{\circ}} E^{\circ}=\hat{F}^{\circ}\widehat{\otimes}_{F^{\circ}} E^{\circ}\cong \hat{E}^{\circ}.$$ This implies that $\Omega_{\hat{E}^{\circ}/\hat{F}^\circ}\cong \Omega_{E^{\circ}/F^\circ} \otimes_{F^{\circ}} \hat{E}^\circ.$ Equality of differents then follows by applying $\mathrm{length}()$ to both sides. The case of log-differents is similar, or alternatively follows from part (a) and the non-logarithmic case.

Part (d): the case where $E^{\circ}/F^{\circ}$ is monogenic follows from \cite[III. 6. Cor. 1]{local fields}. By \cite[III. 6. Proposition 12]{local fields}, the extension $E^{\circ}/F^{\circ}$ is monogenic if either $\widetilde{E}/\widetilde{F}$ is seperable or $E/F$ is of prime power degree and residually inseperable. If $E/F$ is Galois, the fixed field of a $p$-sylow subgroup of $\mathrm{Gal}(E/F)$ splits $E/F$ into a tower of a residually seperable and a prime power degree residually inseperable extension. After passing to a Galois closure, it will therefore suffice to show that both sides of \eqref{eqn: kcan} are compatible in finite seperable towers $E/F/G$ of extensions of discretely valued fields. This follows from the adjunction formulae of (b) and the canonical isomorphisms $\Hom_{F^{\circ}}(E^{\circ},F^{\circ})\otimes_{F^{\circ}}\Hom_{G^{\circ}}(F^{\circ},G^{\circ})\cong \Hom_{G^{\circ}}(E^{\circ},G^{\circ}),$ which are both compatible with trace elements.

Part (e): the equivalence of (ii) and (iii) is discussed in \ref{relation with canonical modules}. 
For the equivalence of (i) and (ii), it is classical that $\Omega_{E^{\circ}/F^{\circ}}=0$ if and only if $E^{\circ}/F^{\circ}$ is unramified. If $\widetilde{E}/\widetilde{F}$ is seperable, it follows from part (a) and \cite[III. Proposition 13]{local fields} that $E^{\circ}/F^{\circ}$ is tamely ramified if and only if $\delta^{\log}_{E/F}=0$. The non-trivial fact to prove is that if $\widetilde{E}/\widetilde{F}$ is inseperable, then $\delta^{\log}_{E/F}>0$. By (c) we can assume $E,F$ are complete, which means $E/F$ is defectless. By part (b) we can use the splitting method as in \ref{rmk: canonical is dualising}, and reduce to the case where $E/F$ is purely residually inseparable. Write $q=[E:F]$, which is a prime power. Choose an element $x\in E^{\circ}$ so that $\widetilde{x}$ generates $\widetilde{E}/\widetilde{F}$. Then $\{1,\widetilde{x},\dots,\widetilde{x}^{q-1}\}$ is a $\widetilde{F}$-basis of $\widetilde{E}$ and it follows that $\{1,x,\dots,x^{q-1}\}$ is an orthogonal $F$-basis of $E$. Let $f(x)$ denote the minimal polynomial of $x$, then $\Omega_{E^{\circ}/F^{\circ}}=E^{\circ}\mathrm{d}x/(f'(x)\mathrm{d}x),$ and so using (a) we have $\delta_{E/F}^{\log}=\delta_{E/F}=v_E(f'(x)).$ We can suppose that $f(x)=x^{q}+a_{q-1}x^{q-1}+\dots+a_1x+a_0$ for some $a_1,\dots,a_{q-1}\in F^{\circ\circ}$ and $a_0\in F^{\circ}$. Since the powers of $x$ form an orthogonal basis, it follows that $v_{E}(f'(x))=\min_{i=1,\dots,q} v_E(ia_i)>0,$ where we let $a_{q}=1$. This completes the proof of (e). See \cite[\S A.3]{AS} for another proof.

Part (f) is shown similarly as part (e): the residually seperable case follows from \cite[III, proposition 13, remark 1.]{local fields}, and for the general case we can use the splitting method, as both sides are insensitive to tame extensions and the right hand side obeys the same addition law as $\delta^{\log}_{E/F}$ with respect to towers of field extensions. So we are reduced to the case where $E/F$ is purely residually inseperable; and in that case the computation in (e) shows $\delta^{\log}_{E/F}=v_E(f'(x))\le v_E([E:F]).$
\end{proof}
\begin{remark} The proof of \ref{propn: (log) differents} (a) can be refined as follows.
	Denote by $\lambda_{E^{\circ}/F^{\circ}}:\Omega_{E^{\circ}/F^{\circ}}\to \Omega^{\log}_{E^{\circ}/F^{\circ}}$ the natural map. Then precisely one of the following two statements hold
	\begin{itemize}
		\item $p$ does not divide $e(E/F)$ and $\lambda_{E^{\circ}/F^{\circ}}$ is surjective with kernel isomorphic to $E^{\circ\circ}/F^{\circ\circ}$,   
		\item $p$ divides $e(E/F)$ and $\mathrm{ker}\lambda_{E^{\circ}/F^{\circ}}\cong E^{\circ}/F^{\circ\circ}$ and $\mathrm{coker}\lambda_{E^{\circ}/F^{\circ}}\cong \widetilde{E}$.
	\end{itemize}
	In particular, the kernel and cokernel of $\lambda_{E^{\circ}/F^{\circ}}$ are always cyclic.
	\begin{proof}
		Equip $\Z$ with the trivial log-structure. The Poincar\'e residue maps for $E^{\circ}$ and $F^\circ$ \cite[IV.1.2.14]{ogus}; 
 fit in the commuting square

$$\begin{tikzcd}
\Omega_{F^{\circ}/\Z}^{\log}\otimes E^{\circ} \arrow[r, ""] \arrow[d, "\mathrm{res}\otimes 1"] & \Omega^{\log}_{E^{\circ}/\Z} \arrow[d, "\mathrm{res}"] \\
E^{\circ}/F^{\circ\circ} \arrow[r, "{\cdot\,{e(E/F)}}"]    & \widetilde{E},                                                  
\end{tikzcd} $$

where the top map is the natural one. By the proof of \cite[\S7.28]{W} we obtain a short exact sequence $$0\lra \ker\lambda_{E^{\circ}/F^{\circ}}\lra E^{\circ}/F^{\circ\circ} \overset{\cdot\,{e(E/F)}}{\lra}    \widetilde{E} \lra \coker \lambda_{E^{\circ}/F^{\circ}} \lra 0,$$ from which the result follows.
	\end{proof}
\end{remark}

\begin{remark} To put \eqref{eqn: kcan} into context, let us explain a generalisation via Grothendieck-Verdier duality theory \label{rmk: duality}  \cite{conrad}. For any syntomic (i.e. flat and l.c.i.) and proper morphism of schemes $f:X\to Y$, there exists an adjunction $f_*:\mathbf{QCoh}(Y)\rightleftarrows \mathbf{QCoh}(X):f^!$ such that $f_*f^!\mathscr{G}=\underline{\Hom}_{\O_Y}(f_*\O_X,\mathscr{G})$, and this determines $f^!\mathscr{G}$ if $f$ is affine. Moreover $f^!\mathscr{G}\cong f^*\mathscr{G}\otimes \omega_{X/Y}$ where $\omega_{X/Y}=\det \Omega_{X/Y}$ is the canonical sheaf (well-defined since $\Omega_{X/Y}$ is perfect by the assumption that $f$ is l.c.i.), in particular $f^!\O_{Y}\cong \omega_{X/Y}$. 
The counit map $f_*f^!\mathscr{G}\lra \mathscr{G}$ in the case where $\mathscr{G}=\O_Y$ yields a \emph{residue} map $$\mathrm{res}_f:f_*\omega_{X/Y}\to \O_Y.$$ 
Now suppose $f:X\to Y$ is finite. Then from linear algebra there is the ``usual'' $\O_Y$-linear trace map $\mathrm{tr}_f:f_*\O_X\to \O_Y$, and thus we get an adjoint map $\mathscr{O}_X\to f^!\O_Y\cong\omega_{X/Y}$ which generalises the trace morphism. The image of $1$ is the \emph{trace element} $\tau_{X/Y}$ of $\omega_{X/Y}$. The trace map $\mathrm{tr}_f$ equals the composition $$f_*\O_X\lra f_*\omega_{X/Y}\overset{\mathrm{res}_f}{\lra} \O_{Y},$$ see for instance \cite[Prop. 1]{Fargues}, and so we recover Proposition \ref{propn: (log) differents} (d). 
For a further generalisation (due to Tate) to any quasi-finite syntomic $f:X\to Y$, see \cite[0BWH]{stacks}. 
\end{remark}
\section{Models and skeleta of curves}
\label{sec: models}
\begin{block}[Models]\label{block models}
Recall our conventions of $\ref{notation}$ on $k^{\circ}$-models $\CC$ of $C/k$. In particular, models are assumed to be proper and normal, and by default we equip $\CC$ with the divisorial log-structure defined by the special fiber $\CC_s$. We call a $k^{\circ}$-model $\CC$ \emph{toroidal} if $\CC^{\dagger}$ is logarithmically regular for the Zariski topology \cite{K94}. We call a $k^{\circ}$-model $\CC$ \emph{snc} if $\CC$ is regular and $\CC_{s,\red}$ is a strict normal crossings divisor in $\CC$. In fact \emph{snc} $k^{\circ}$-models coincide with regular toroidal $k^{\circ}$-models by \cite[\S11.6]{K94}. 
\end{block}
\begin{block}[Contraction of components]
\label{block: contraction} 
By the theory of contraction of normal fibered surfaces, see \cite[\S6.7/4]{BLR}, we have the following. 
Let $\CC$ be a $k^{\circ}$-model and let $C_i$, $i\in I$ be the irreducible components of the special fiber $\CC_s$. Let $J$ be a \emph{strict} subset of $I$. Then there exists a unique proper \emph{contraction morphism} $f:\CC\to \CC_J$ to a $k^{\circ}$-model $\CC_J$ such that $C_j$, $j\in J$ are precisely the components of $\CC_s$ which map onto a closed point, and $f$ defines an isomorphism on $\CC\setminus \cup_{j\in J}C_j$. We will also call $\CC_J$ a \emph{contraction} of $\CC$.
\end{block}
\begin{block}[Relatively minimal models]\label{minimal models}
A $k^{\circ}$-model $\CC$ is called \emph{relatively minimal snc} if $\CC$ is \emph{snc} and all contractions in the sense of \ref{block: contraction} are no longer \emph{snc}. It is well-known that if $g(C)\ge 1$ then every relatively minimal \emph{snc} model is minimal \emph{snc}, i.e. final in the category of \emph{snc} $k^{\circ}$-models, see \cite[\S9.3.20]{Liu} and erratum \cite{liu erratum}. The proof of loc.cit also implies the following variant, not assuming $g(C)\ge 1$. Fix a given $k^{\circ}$-model $\CC$ and only consider \emph{snc} $k^{\circ}$-models which dominate $\CC$. Then $k^{\circ}$ models relatively minimal \emph{snc} over $\CC$ are minimal \emph{snc} over $\CC$, i.e. final among \emph{snc} $k^{\circ}$-models over $\CC$.
\end{block} 
\begin{remark}
The theory of minimal models of \ref{minimal models} does not extend verbatim by replacing '\emph{snc}' by 'toroidal' throughout: one must consider log-regularity in the \'etale topology to deal with self-intersections of components, and require $g(C)\ge 2$; see \cite{Nag}. \end{remark}
\begin{block}[Resolutions]
By Lipman's resolution of $2$-dimensional excellent schemes \cite{lipman} any $k^{\circ}$-model admits a \emph{resolution} , i.e. a morphism $h:\CC'\to\CC$ of $k^{\circ}$-models such that $\CC'$ is a \emph{snc} $k^{\circ}$-model and $h^{-1}\CC_{\mathrm{reg}}\overset{\cong}{\to}\CC_{\mathrm{reg}}$. Then by \ref{minimal models} there exist a unique minimal \emph{snc} $k^{\circ}$-model over $\CC$, also called \emph{the minimal snc resolution} of $\CC$. 
\end{block}
\begin{block}[Toroidal resolutions]
Now assume that $\CC$ is a toroidal $k^{\circ}$-model, then the fan of $\CC$ is $2$-dimensional and therefore admits a unique minimal regular subdivision, giving rise to the minimal \emph{snc} resolution $h:\CC'\to \CC$ \cite[\S10-11]{K94} (also see \cite[\S A.29]{W}), it is explicitly described as follows. Let $x\in \CC$ be a singular point, note that $x$ is a closed point since $\CC$ is normal. Then the exceptional divisor $E_x\coloneqq h^{-1}{x}$ above $x$ satisfies condition \emph{(chain)} below.
\begin{enumerate}
  \item[\emph{(chain)}] The divisor $E_x$ is a strict normal crossings divisor, and can be written as a sum $\sum_{i=1}^n E_i$ of prime divisors such that $E_i\cong \P^1_{k(x)}$ and $E_i^2\le -2$ for all $i$. 
   Also for all $1\le i\ne j\le n$ we either have $|i-j|>1$ and $E_i\cap E_j=\emptyset$ or $|i-j|=1$ and $E_i\cap E_j$ is a unique $k(x)$-rational point. The divisor $E$ intersects the strict transform of $\CC_s$ in exactly two $k(x)$-rational points, one on $E_1$ and one on $E_n$. 
\end{enumerate}
\end{block}
\begin{remark} The $\kappa(x)$-rationality conditions in (\emph{chain}) can be omitted, by the assumption that $\tilde{k}$ is algebraically closed. We formulated (\emph{chain}) such that Lemma \ref{lem: characterisation toroidal models of curves} below is also true without this assumption. This lemma is essentially due to Ito and Schro\"er.  
\end{remark}		
\begin{lemma}[Ito-Schro\"er]
  \label{lem: characterisation toroidal models of curves} A $k^{\circ}$-model $\mathscr{C}$ of $C$ is toroidal if and only if there exists a resolution $h:\CC'\to \CC$ such that the exceptional divisor $E_x$ above each point $x\in \CC$ satisfies condition (\emph{chain}).
\end{lemma}
\begin{proof} Suppose $E_x$ satisfies \emph{(chain)}. An application of Lipman's generalisation of Artin's contractibility criterion \cite[27.1]{lipman rational} shows that $x$ is a rational singularity (the fundamental cycle equals $E_x$, so $\chi(E_x)=1>0$). Hence the argument of \cite[\S3.4]{IS} shows that the log-structure of $\CC^{\dagger}$ is \emph{fs} at $x$ and by the proof \cite[\S3.6]{IS} it follows that $\CC$ is toroidal at $x$. 
\end{proof}
\begin{block}[Analytifications]\label{comb study of skeleta assumptions}
  Write $C^{\an}$ for the Berkovich $k$-analytification of $C$ \cite{B90}. Its points $x=|\cdot(x)|\in X$ come in three types:
  \begin{enumerate}
   	\item[(type 1)] closed points of $C$, 
   	\item[(type 2)] discrete $k$-valuations on $k(C)$,
   	\item[(type 3)] non-discrete $k$-valuations on $k(C)$.
   \end{enumerate} 
   There are no type 4 points since $k$ is spherically complete.
   Recall that the completed residue field at a point $x\in X$ is denoted by $\H(x)$. For a type $2$ point $x\in X$ the residual curve $C_x$ is defined as the unique proper smooth $\tilde{k}$-curve with function field $\widetilde{\H(x)}$ and we denote $g(x)$ for the genus of $C_x$. If $x\in X$ is not of type $2$ we let $g(x)=0$. We write $\chi(x)=2-2g(x)$. A branch at a point $x$ is an equivalence class of intervals emanating from $x$, and the number of branches is 1,$\infty$,2 for type 1,2,3 points respectively, see \cite{ducros} for a thorough study of Berkovich curves.
\end{block}
\begin{block}[Dual graph] Let $\CC$ be a toroidal $k^{\circ}$-model. We write $V(\CC)$ for the \emph{vertex set} associated to $\CC$, defined as the set of generic points of $\CC_{s}$. Given $x\in V(\CC)$ we write $C_x\coloneqq \overline{\{x\}}$ for the corresponding component of $\CC_s$. The \emph{dual graph} $\Gamma(\CC)$ of $\CC$ is defined as the graph (undirected, loopless, possibly some multiple edges) with vertex set $V(\CC)$ and $x,x'\in V(\CC)$ are connected by an edge $[x,x']$ if and only if $C_x\cap C_{x'}\ne\emptyset$.  

We can also view a point $x\in V(\CC)$ as a discrete $k$-valuation on $k(C)$, namely the valuation defined by $f\mapsto \frac{1}{m(x)}\ord_{C_x}(f)$ where $m(x)\coloneqq \mathrm{mult}_{\CC_s}C_x$ denotes the \emph{multiplicity} of $x$. So we can view $V(\CC)$ as a finite set of type $2$ points of $C^{\an}$. 
\end{block}
\begin{remark}In fact, $\CC$ can be reconstructed from its vertex set $V(\CC)\subset C^{\an}$ (we omit some details since this remark will be of no further use): if $S\subset C^{\an}$ is a non-empty finite set of type $2$ points of $C^{\an}$, then via a blow-up procedure there exists a $k^{\circ}$-model $\CC$ such that $V(\CC)\supset S$, and by \ref{block: contraction} the unique contraction $\CC\to\CC_J$ where $J=S\setminus V(\CC)$ satisfies $V(\CC_J)=S$. This gives a one-to-one correspondence between $k^{\circ}$-models and non-empty finite sets of type $2$ points contained in $C^{\an}$ such that $\CC_1$ dominates $\CC_2$ if and only if $V(\CC_1)$ contains $V(\CC_2)$. A generalisation can be found in \cite[6.3.15]{ducros}. \label{correspondence vertex sets and proper models}\end{remark}
\begin{definition}[Skeleta]\label{definition skeleton} Let $\CC$ be a toroidal $k^{\circ}$-model. By \cite[\S3]{BM}, generalising earlier constructions of Berkovich and others, there exists a canonical subspace $\sk(\CC)\subset C^{\an}$ called the \emph{skeleton of $\CC$}, which is a geometric realisation of $\Gamma(\CC)$ identifying points of $V(\CC)$ with their corresponding type $2$ point. Moreover there exists a canonical continuous retraction $\rho_{\CC}:C^{\an}\to \sk(\CC)$ and the argument of \cite[\S5.2]{BPR} shows that the induced map $C^{\an}\to \lim \sk(\CC)$ is a homeomorphism.

  We call a subspace $\Gamma\subset X$ a \emph{skeleton of $X$} if $\Gamma=\sk(\CC)$ for some toroidal model $\CC/k^{\circ}$ of $C/k$. Note that any skeleton is also the skeleton of a strict normal crossings (\emph{snc}) model $\CC$ because a toroidal resolution does not change the skeleton. 

  We equip $C^{\an}$ with the potential metric $d(\cdot,\cdot)$ introduced in \cite{BN}. On a skeleton $\Gamma$ this metric is induced by the \emph{conformal} measure $\mu_{\Gamma}$ introduced in \cite[\S4]{W}. 
\end{definition}

\begin{definition}[subgraphs of $C^{\an}$]\label{metric graphs}\label{subgraph} In this paper we define a \emph{subgraph} $\Gamma$ of $X$ as a compact connected subspace which is the geometric realisation of a finite (loopless, undirected, possibly multiple edges) graph with type $2$ leaves. 
  Let $\Gamma$ be a subgraph of $X$. For any point $x\in \Gamma$ we denote by $\br_\Gamma(x)$ the set of \emph{branches at $x$ in $\Gamma$}, i.e. germs of intervals of $\Gamma$ emanating from $x$. 
  We write $$\val_\Gamma(x)\coloneqq \#\br_\Gamma(x)$$ for the \emph{valency} of $x$ with respect to $\Gamma$, note that $\val_\Gamma(x)$ is finite for all $x\in \Gamma$. We write \begin{equation}
\chi_\Gamma(x)=\chi(x)-\mathrm{val}_\Gamma(x).
\end{equation}

  We define the set of \emph{vertices} $V(\Gamma)$ of $\Gamma$ by $$V(\Gamma)\coloneqq\{x\in\Gamma:\chi_\Gamma(x)\ne0\}=\{x\in \Gamma:\ \mathrm{val}_\Gamma(x)\ne 2\text{ or }g(x)>0\}.$$ Note that $V(\Gamma)$ is a finite set of type $2$ points since $\Gamma$ is assumed compact. Beware that for a toroidal $k^{\circ}$-model the inclusion $V(\sk(\CC))\subset V(\CC)$ can be strict.
The set of \emph{leaves of $\Gamma$} is the set $\leaf(\Gamma)=\{x\in \Gamma:\val_\Gamma(x)=1\}$. We define the combinatorial divisor $K_{\Gamma}=\sum_{x\in\Gamma}m(x)\chi_\Gamma(x)[x]$, where $m(x)=e(\H(x)/k)$ is the \emph{multiplicity} of a type $2$ point, and the 
 \emph{Euler characteristic of $\Gamma$} is defined as \begin{equation}
  \label{eqn: chi of skeleton} \chi(\Gamma)\coloneqq\deg K_\Gamma.\end{equation} 
  Note that by our assumptions only finitely many summands are nonzero, and hence $\chi(\Gamma)$ is well-defined.
\end{definition} 
\begin{block}\label{blowup sequence}
	Suppose $\CC$ and $\CC_0$ are \emph{snc} $k^{\circ}$-models such that $\CC$ dominates $\CC_0$. Then by the factorisation theorem, see \cite[\S9.2.2]{Liu}, the morphism $\CC\to\CC_0$ factors as a composition of blowups $$\CC=\CC_n\to\CC_{n-1}\to\dots\to \CC_0$$ in closed points. At each step, there are two possibilities: \begin{enumerate}
  \item \emph{Blowup of a nodal point $z\in(\CC_{i})_s$}. In this case $\sk(\CC_{i+1})=\sk(\CC_i)$, and the exceptional component of $\CC_{i+1}\to\CC_i$ has multiplicity equal to $\mathrm{mult}(E_1)+\mathrm{mult}(E_2)$ where $E_1$ and $E_2$ are the components of $(\CC_i)_s$ containing $z$ (possibly $E_1=E_2$).
  \item \emph{Blowup of a smooth point $z\in (\CC_{i})_{s,\red}$}: in this case, $\sk(\CC_{i+1})$ is obtained by gluing an interval to $\sk(\CC_i)$ and the exceptional component of $\CC_{i+1}\to\CC_i$ has multiplicity $\mathrm{mult}(E)$ where $E$ is the unique component of $(\CC_i)_s$ containing $z$. 
\end{enumerate} 
\end{block}

\begin{definition}\label{defn neat}
  Let $I=(x,y)$ be a subinterval of $C^{\an}$. Then we call $I$ \emph{neat} if there exists a skeleton $\Gamma=\sk(\CC)$ such that $I\in \pi_0(\Gamma\setminus V(\CC))$ and $m(x)=m(y)$.
\end{definition}
\begin{example} For instance, suppose that $\Gamma$ is a skeleton and let $x\in C^{\an}\setminus \Gamma$ be a divisorial point. Suppose that $m(\rho_{\Gamma}(x))=m(x)$ and $m(x)^2d(x,y)\in\Z$. Then $I=[x,\rho_{\Gamma}(x))$ is neat: indeed, we can suppose $\Gamma=\sk(\CC)$ for some proper \emph{snc} model $\CC$ such that $\rho_{\Gamma}(x)\in V(\CC)$, and by repeatedly blowing up the specialisation of $x$ we obtain a proper \emph{snc} model $\CC'$ such $V(\CC')=V(\CC)\sqcup \{x\}$ and $I=\sk(\CC')\setminus (\{x\}\cup\sk(\CC))$.\label{specialisation neat}
\end{example}
\begin{lemma}[Combinatorial characterisation of multiplicities]
  \label{nc interval multiplicities}
  Let $I=(x,y)\subset C^{\an}$ be a neat interval such that $d(x,y)=m(x)^{-2}$. Let $z\in I$ be divisorial. Then $m(z)$ is equal to the smallest integer $M$ for which $M\cdot m(x)d(x,z)\in \Z$. 
\end{lemma}
\begin{proof} For each $q\in [0,1]\cap \Q$ let $(a(b),b(q))$ be the unique pair of coprime integers such that $q=a(q)/b(q)$. Let $S_M=\{q\in [0,1]\cap\Q: b(q)\le M\}$. By the theory of Farey sequences it is well known that if $0=q_{0}<q_{1}<\dots <q_{r}=1$ are the elements of $S_M$ then $q_{i+1}-q_i=\frac{1}{b(q_{i})b(q_{i+1})}$.

By the assumptions we can choose an \emph{snc} model $\CC$ such that $x,y\in V(\CC)$ and $I\in\pi_0(\sk(\CC)\setminus V(\CC))$. Consider the unique isometry $\phi:[0,1]\to I$ such that $\phi(0)=x$ and $\phi(1)=y$. Then by induction on $M$, it follows with the above notation that $m(\phi(q_i))=b(q_i)m(x)$. Therefore $m(\phi(q))=b(q)m(x)$ for all $q\in [0,1]\cap\Q$ and the result follows because $d(x,\phi(q))=q/m(x)^2$.
\end{proof}
\begin{proof}[Proof of Theorem \ref{thm: enlarging skeleta}]
First, assume that $\Gamma$ is a skeleton. We can assume that $\Gamma=\sk(\CC)$ for some \emph{snc} model $\CC$ of $C$. By a intersection-theoretic computation, see for instance \cite[Lemma 3.1.2]{N12}, it follows that $\chi(\Gamma)=\chi(C)$. For the other direction, let us consider the unique  \emph{snc} model $\CC$ which is minimal with respect to the property that $$\Gamma\subset \sk(\CC).$$ 
Note that $\sk(\CC)$ is the smallest skeleton containing $\Gamma$. By \ref{blowup sequence} we can factor $\CC\to\CC_{\min}$ as a sequence of blowups in closed points. Observe that in the blowup sequence, new nodes of $\sk(\CC)$ can only be created by blowing up smooth points of the reduced special fiber of components that are \emph{not} leaves.  
It follows by minimality of $\CC$ that \emph{any node of $\sk(\CC)$ is contained in $\Gamma$ 
and no node of $\sk(\CC)$ is a leaf of $\Gamma$}. 
Hence $\sk(\CC)\setminus \Gamma$ is a disjoint union of intervals, that is 
 $$\sk(\CC)=\Gamma \cup \bigsqcup_{x\in \leaf(\sk(\CC))}[x,\rho_\Gamma(x)].$$ Note that for all $x\in \leaf(\sk(\CC))$ we have that $\rho_{\Gamma}(x)$ is of type $2$ and moreover $m(\rho_{\Gamma}(x))\ge m(x)$ since $\rho_{\Gamma}(x)$ can be obtained as a divisorial point after a sequence of nodal blowups of $\CC$. 
  Now observe that for all $x\in\leaf(\sk(\CC))$ we have that $$\chi_{\sk(\CC)}(x)=1\quad\text{and}\quad\chi_{\sk(\CC)}(\rho_{\Gamma}(x))=\chi_{\Gamma}(\rho_{\Gamma}(x))-1$$ and therefore $$0\le \sum_{x\in \leaf(\sk(\CC))}m(\rho_\Gamma(x))-m(x)=\chi(\Gamma)-\chi(\sk(\CC)).$$ By assumption $\chi(\Gamma)=\chi(\sk(\CC))$ It follows that $m(\rho_\Gamma(x))=m(x)$ is an equality for every $x\in \leaf(\sk(\CC))$, and as in \ref{specialisation neat} it follows by minimality of $\CC$ that $\Gamma=\sk(\CC)$, as desired.
\end{proof}

\section{Skeleta and base change}
\label{sec: bc}
\begin{block}\label{ass sec bc}
	In this section we specialise the results of \cite{W} to the following setting. Let $C/k$ as in \ref{notation}, let $\CC/k^{\circ}$ be a toroidal model and let $k'/k$ be a finite seperable extension. We make the hypothesis that the normalised base change $\CC'\coloneqq (\CC\otimes_{k^{\circ}}k'^{\circ})'$ is a toroidal $(k')^{\circ}$-model of $C'\coloneqq C\otimes_{k}k'$. Note that $\CC'$ is also a toroidal $k^{\circ}$-model of $C'$ viewed as a (non-geometrically connected) $k$-curve. We write $\Gamma'=\sk(\CC')$ and $\Gamma=\sk(\CC)$. Let $\pi:C'\to C$ be the projection, and let $\phi=\pi^{\an}\vert_{\Gamma'}:\Gamma'\to\Gamma$ be the induced base change map of skeleta. 
\end{block}
\begin{proposition}[Balancing condition]\label{balancing} For all edges $I'\in \pi_0(\Gamma'\setminus V(\CC'))$, write $\deg_{I'}\phi=\mathrm{length}(\phi(I))/\mathrm{length}(I')$. Then $I'\to \phi(I)$ is linear in the sense that $\phi_*\mu_{I'}=\frac{1}{\deg_I'\phi}\mu_I$. Moreover for all $I\in \pi_0(\Gamma\setminus V(\CC))$ we have $$[k':k]=\sum_{I'\in \pi_0(\phi^{-1}I)}\deg_{I'}\phi.$$
\end{proposition}
\begin{proof}
	This follows from \cite[\S3.24, \S4.7]{W} and the fact that $\deg\pi=[k':k]$.\end{proof}
	\begin{remark} It follows moreover from \cite[\S3.24]{W} that for each $x'\in \Gamma'$ and each edge $I$ as in \ref{balancing} incident to $\phi(x)$ the sum $\deg_x\phi\coloneqq \sum_{I'/I}\deg_{I'}\phi$, summing over edges $I'$ above $I$, is independent of the choice of $I$. In fact $\deg_{x'}\phi=[\H(x'):\H(x)]=\deg_{x'}\pi^{\an}$ by \cite[[\S3.5.4.2 \& \S4.2.3(iii)]{ducros}.
	\end{remark}
	\begin{block} Given a point $x\in \Gamma$, we write $\mathrm{Br}_{\Gamma}(x)\coloneqq \colim_{U}\pi_0(U\setminus \{x\})$ where the colimit runs over neighbourhoods of $x$, for the set of \emph{branches} at $x$, i.e. the set of germs of intervals emanating from $x$. Given a piecewise linear function $F$ on $\Gamma$ and a branch $b\in\mathrm{Br}_{\Gamma}(x)$ at a point $x\in \Gamma$, we write $\partial_bF$ for the outgoing slope of $F$ in the direction of $b$ and we write $\Delta(F)=\sum_{x\in \Gamma}\sum_{b\in\mathrm{Br}_{\Gamma}(x)}\partial_bF[x]$ for the combinatorial Laplacian of $F$, which is a combinatorial divisor on $\Gamma$, i.e. a formal linear combination of points on $\Gamma$. 
We define a pullback of combinatorial divisors by letting $\phi^*[x]=\sum_{x'\in\phi^{-1}(x)}\deg_{x'}\phi[x']$ and extending linearly.
	\end{block}
	\begin{proposition}[Different function for base change of skeleta]\label{rh} Assumptions as in \ref{ass sec bc}. Let $\delta:\Gamma'\to\R_{\ge0}$ be the unique function which is affine-linear on each edge $I\in \pi_0(\Gamma'\setminus V(\CC'))$ and satisfies $$\delta(x')=\frac{1}{m(x')}\delta^{\log}_{\O_{\CC',x'}/\O_{\CC,\pi(x)}}$$ for all $x'\in V(\CC')$.
	Then:
	\begin{enumerate}
		\item For each type $2$ point $x'\in \Gamma'$ we have $\delta(x')=\frac{1}{m(x')}\delta^{\log}_{\H(x')/\H(x)}$ where $x=\pi^{\an}(x')$,
		\item $\delta$ has integral slopes,
		\item we have the Riemann-Hurwitz formula $\Delta(\delta)=K_{\Gamma'}-\phi^*K_{\Gamma}$.
	\end{enumerate}		
	\end{proposition}
	\begin{proof}
		 \cite[\S8.7]{W} and \ref{propn: (log) differents} (c).
	\end{proof}
	\begin{proposition} For all $x'\in \Gamma'$ we have $\delta(x)\le v_{k}[k':k]$.\label{bounds different}
	\end{proposition}
	\begin{proof}
		This follows from \ref{propn: (log) differents} (f).
	\end{proof}
	\begin{block} Any type $2$ point $x$ admits a seperable parameter \cite[\S8.3]{T16}, i.e. an element $t\in (H(x)^{\circ})^\times$ such that $\H(x)/k(t)$ is residually seperable (consider any lift of a seperable transcendence basis for $\widetilde{\H(x)}/\tilde{k}$). A parameter $t$ is called tame if $\H(x)/k(t)$ is tamely ramified. Since we have an equality of ramification indices $e(\H(x)/k(t))=e(\H(x)/k)=m(x)$, tame type $2$ points are the same as type $2$ points which satisfy $p\nmid m(x)$. The closure of the tame type $2$ points of $\Gamma$ is called the \emph{temperate part} $\Gamma^t$ of $\Gamma$ \cite[\S4.3]{JN}.
	\end{block}
\begin{proposition} 
		\label{prop: temp} \label{constant along temperate part} The function $\delta\vert_{\phi^{-1}\Gamma^{t}}$ is constant and equal to $\frac{1}{[k':k]}\delta_{k'/k}^{\log}$.
\end{proposition}
\begin{proof} 
Suppose $x\in \Gamma$ be a tame divisorial point. After a subdivision we can suppose that $x\in V(\CC)$.   
Choose a parameter $t\in\O_{\CC,x}^\times$, i.e. an element such that $\O_{\CC,x}/k(t)$ is tamely ramified. Then the normalised base change $\O_{\CC',x}/k'(t)^{\circ}$ is tamely ramified too. Now apply \ref{propn: (log) differents} (b) to the towers $\O_{\CC',x'}/\O_{\CC,x}/k(t)^{\circ}$ and $\O_{\CC',x'}/k'(t)^{\circ}/k(t)^{\circ}$, and use that $\Omega_{k'^{\circ}/k^{\circ}}^{\log}\otimes_{k^{\circ}}k(t)^{\circ}\cong \Omega_{k'(t)^{\circ}/k(t)^{\circ}}^{\log}$.
\end{proof}
\begin{remark} An argument similar to the proof of Proposition \ref{constant along temperate part} gives the following relation between the weight functions of \cite{MN} over $k'$ and $k$: $$\wt^{k'}_{\omega'}(x)=[k':k]\wt^k_{\omega}(\pi(x))+\delta^\log_{k'/k}.$$ When $k'/k$ is tame this recovers \cite[Proposition 4.5.1]{JN}.
\end{remark}
\begin{proposition}[Comparison with classical Riemann-Hurwitz formula]\label{comparison RH} Suppose that $x'\in \Gamma'$ is a type $2$ point such that $\H(x')/\H(x)$ is residually seperable, where $x=\phi(x')$. Let $b'\in \mathrm{Br}_{\Gamma'}(x')$ and write $b=\phi(b')\in \mathrm{Br}_{\Gamma}$. Let $C_{x'}\to C_x$ be the associated map of residual curves and let $\tilde{b'}\in C_{x'}$ and $\tilde{b}\in C_x$ be the points corresponding to $b'$ and $b$ respectively. Then we have \begin{equation}
    \partial_{b'}\delta=m(x')\delta^{\log}_{\O_{C_{x'},\tilde{b'}}/\O_{C_{x},\tilde{b}}}.\label{eqn: comparison rh}
  \end{equation}
\end{proposition}
\begin{proof} By the Poincar\'e-Lelong formula of \cite[\S6.19]{W} and its proof one has $\partial_{b'}\delta=m(x')\ord_{\tilde{b'}}\sp_*\tau_{\CC'/\CC}$
where the trace section $\tau_{\CC'/\CC}\in \Gamma(\CC',\omega_{\CC'^{\dagger}/\CC^{\dagger}})$ is viewed as a section the logarithmic canonical bundle, see \cite[\S7.33]{W}. By the computation \cite[\S7.32]{W} it follows that $$\sp_*\omega_{\CC'^{\dagger}/\CC^{\dagger}}=\omega_{C_{x'}^\dagger/C_x^{\dagger}},$$ and moreover $\sp_*\tau_{\CC'/\CC}=\tau_{C_{x'}/C_{x}}$ is the trace element of $\omega_{C_{x'}/C_x}$ (here we use the assumption that $C_{x'}\to C_x$ is seperable). By the comparison formula of \cite[\S7.28]{W} it follows $$\ord_{\tilde{b'}}\sp_*\tau_{\CC'/\CC}=\ord_{\tilde{b}}\tau_{C_x/C_{f(x)}}+1-e({\O_{C_x,\tilde{b}}/\O_{C_{f(x)},f(\tilde{b})}})$$
  By \ref{2 descriptions diff}(ii) and \ref{propn: (log) differents}(a,e) the latter equals $\delta^{\log}_{\O_{C_x,\tilde{b}}/\O_{C_{f(x)},f(\tilde{b})}}$.
\end{proof}
\begin{block}[Galois degree $p$ case]\label{galois deg p case}
Some simplifications arise if $k'/k$ is assumed Galois of degree $p$. Let $x\in \Gamma$, then we have two cases: $\#\phi^{-1}(x)=p$ or $\#\phi^{-1}(x')=1$, that is $x$ splits in $p$ preimages or topologically ramifies. Note that $\H(x')/\H(x)$ is trivial if and only if $\pi^{\an}$ is a local isomorphism at $x'$, by \cite[Theorem 3.4.1]{B93}. Then $\phi$ is an isometry on the splitting locus, and the metric expands by a factor $p$ on the topological ramification locus. Beware that the metric on $(C')^{\an}$ is computed with respect to $k$ and not $k'$ (see for instance Example \ref{example II}). A case-by-case analysis using \ref{propn: (log) differents} (e) shows that the interior of the topological ramification locus is the locus where $\delta>0$. 
\end{block}

\section{Reduction types of elliptic curves}
\label{sec:ell curve}
\begin{block}
Throughout this section, $E$ is an elliptic curve, its neutral element is denoted by $0\in E(k)$. Let $\mathscr{E}_{\min}$ be the minimal \emph{snc} model and write $\Gamma_{\min}\subset E^{\an}$ for its skeleton.
The combinatorial types of $\Gamma_{\min}$ are classified by the Kodaira-N\'eron classification, summarised in Table \ref{kodaira neron} below, also see \cite[Theorem 8.2]{silverman advanced topics} for details. In Table \ref{kodaira neron}, we have moreover indicated the vertices $V(\Gamma_{\min})$ together with their multiplicity. 

 Let $\omega$ denote a choice of \emph{invariant differential}, i.e. a generator of the $1$-dimensional $k$-vector space $H^0(E,\omega_{E})$; recall $\div\omega_E=0$. The minimum locus of the weight function $\wt_{\omega_{E}}$ introduced in \cite{MN} is called the essential skeleton and denoted by $\sk(E)$, it is coloured red in Table \ref{kodaira neron}. Beware that the essential skeleton is not a skeleton in the sense of this paper. The weight function $\wt_{\omega_E}$ is \emph{strictly increasing} away from $\sk(E)$, see \cite[\S4.4.5]{MN}.
\end{block}
\refstepcounter{equation}
\begin{table}[ht!]\label{kodaira neron}
\centering
\begin{tabular}{c|c|c}
\emph{Reduction type} & \emph{Kodaira symbol} & \emph{Minimal skeleton} $\Gamma_{\min}$\\
\hline 
Good reduction & $I_0$ &  \begin{tikzpicture}\draw[fill=red] (0,0) circle (2pt);
\node at (0.25,0.25) {$1$};\end{tikzpicture}  \\
\hline
Multiplicative reduction & $I_{n}$ &   
	\begin{tikzpicture}\draw[red,thick] (0,0) circle (0.7);
	\end{tikzpicture}
 
   \\
\hline
Additive reduction &$I_{n}^*$&  	
 \begin{tikzpicture}
	\node at (2.25,0.25) {$1$};
	\node at (-1.75,0.25) {$1$};
	\node at (1.25,0.25) {$2$};
	\node at (-0.75,0.25) {$2$};
	\node at (1.25,1.25) {$1$};
	\node at (-0.75,1.25) {$1$};
	\draw[fill] (2,0) circle (2pt);
	\draw[fill] (1,0) circle (2pt);
	\draw[fill] (1,1) circle (2pt);
	\draw[fill] (-2,0) circle (2pt);
	\draw[fill] (-1,0) circle (2pt);
	\draw[fill] (-1,1) circle (2pt);
	\draw[thick] (2,0)--(1,0)--(1,1);
	\draw[thick] (-2,0)--(-1,0)--(-1,1);
	\draw[thick,red] (-1,0) -- (-0.5,0);
	\draw[thick,red] (1,0) -- (0.5,0);
	\draw[thick,red,dashed] (-0.5,0)--(0.5,0);
	\draw[fill=red] (-1,0) circle (2pt);
	\draw[fill=red] (1,0) circle (2pt);
\end{tikzpicture}  
 \\
&$II,\, II^*$&  	
 \begin{tikzpicture}
	\node at (0.25,0.25) {$6$};
	\node at (1.25,0.25) {$1$};
	\node at (-0.75,0.25) {$2$};
	\node at (0.25,1.25) {$3$};
	\draw[fill] (-1,0) circle (2pt);
	\draw[fill] (1,0) circle (2pt);
	\draw[fill] (0,1) circle (2pt);
	\draw[thick] (-1,0) -- (1,0);
	\draw[thick] (0,0) -- (0,1);
	\draw[fill=red] (0,0) circle (2pt);
\end{tikzpicture}  
 \\
&$III,\, III^*$&   \begin{tikzpicture}
	\node at (0.25,0.25) {$4$};
	\node at (1.25,0.25) {$1$};
	\node at (-0.75,0.25) {$1$};
	\node at (0.25,1.25) {$2$};
	\draw[fill] (-1,0) circle (2pt);
	\draw[fill] (1,0) circle (2pt);
	\draw[fill] (0,1) circle (2pt);
	\draw[thick] (-1,0) -- (1,0);
	\draw[thick] (0,0) -- (0,1);
	\draw[fill=red] (0,0) circle (2pt);
\end{tikzpicture}   \\
&$IV,\, IV^*$&  
 \begin{tikzpicture}
	\node at (0.25,0.25) {$3$};
	\node at (1.25,0.25) {$1$};
	\node at (-0.75,0.25) {$1$};
	\node at (0.25,1.25) {$1$};
	\draw[fill] (-1,0) circle (2pt);
	\draw[fill] (1,0) circle (2pt);
	\draw[fill] (0,1) circle (2pt);
	\draw[thick] (-1,0) -- (1,0);
	\draw[thick] (0,0) -- (0,1);
	\draw[fill=red] (0,0) circle (2pt);
\end{tikzpicture}\\\\\hline

\end{tabular}
	\caption{\emph{Kodaira-N\'eron classification. For each reduction type, we have drawn the minimal skeleton $\Gamma$. The vertices of $\Gamma$ are marked, together with their multiplicity. The essential skeleton is coloured in red, it is the locus where $\wt_{\omega}$ is minimal.}}
\end{table}

\begin{example}[A $2$-adic elliptic curve with potentially good supersingular reduction.] \label{example II}
Let $E$ be the elliptic curve over $k=\Q_2^{\mathrm{ur}}$ with affine equation $y^2=x^3+2$. This is a minimal Weierstrass equation whose reduction mod $2$ yields the cusp $\tilde{y}^2=\tilde{x}^3$, and by Tate's algorithm $E$ has Kodaira reduction type $II$. The minimal skeleton $\Gamma$ as in Table \ref{kodaira neron} is computed by the embedded resolution of a cuspidal singularity, see for instance \cite[Figure 17, p.406]{Liu}. 
Over the wild Kummer extension $k'=\Q_2(\alpha)$ where $\alpha^2=2$, $E$ ecquires good reduction: indeed, applying the coordinate change $(x_1,y_1)=(x/2,\frac1{2\alpha}(y-\alpha))$ we obtain that $E_{k'}$ admits the affine equation $y_1(y_1+1)=x_0^3$, which is the affine equation of a proper \emph{smooth} model of $E_{k'}$ over $(k')^{\circ}=\Z_2[\alpha]$. 

\paragraph*{Step 1: Computation of the different $\delta_{k'/k}$.}
The different $\delta_{k'/k}$ is computed as follows. Write $G=\mathrm{Gal}(k'/k)$ and denote by $G=G_{-1}\supset G_0\supset G_1\supset \dots$ the lower numbering ramification filtration with unique ramification jump\footnote{Recall that $G_i$ is defined as $G_i=\{\sigma\in G:v_{k'}(\sigma(\varpi')-\varpi')\ge i+1\}$. We have that $G_0$ is the inertia group, and $G_1$ is the wild inertia group. Since $\widetilde{k}$ is algebraically closed we have $G=G_0$. Recall that $s$ is called a \emph{jump} of the higher ramification filtration if $G_{s+1}\subset G_s$ is a proper inclusion. So if $G$ has prime order, there is a unique jump.} $j\ge 1$. 
%

An application of Hilbert's formula, see for instance \cite[IV.1]{local fields}, shows that $$\delta_{k'/k}=\sum_{i\ge 0} \left(|G_i|-1\right)=(p-1)(j+1)=j+1.$$ Since $G=\{1,\sigma\}$ where $\sigma$ is complex conjugation it follows that  $$\delta_{k'/k}=j+1=v_{k'}(\sigma(\alpha)-\alpha)=3,$$ and therefore $\delta_{k'/k}^{\log}=\delta_{k'/k}-e(k'/k)+1=3-2+1=2$.

\paragraph*{Step 2: Construction of a simultaneous skeleton.}
Let $z=\sp_{\mathscr{E}_{\min}}(0,\alpha)$ be the specialisation of the point $(x,y)=(0,\alpha)$ to $\mathscr{E}_{\min}$. Then $z$ lies on the unique component of multiplicity $2$. Let $$\mathscr{X}\to\mathscr{E}_{\min}$$ denote the blowup of $\mathscr{E}_{\min}$ in $z$. A computation of the blow-up algebra shows that $x_1=x/2$ is a local coordinate function on the exceptional component. The exceptional component has multiplicity two by \cite[9.2.23]{Liu}. Let us write $x_0\in E^{\an}$ for the divisorial point associated to this exceptional component.

View $E_{k'}$ as a $k$-curve and let $\pi:E_{k'}\to E$ denote the projection. By the previous points the unique point $x_0'\in E^{\an}_{k'}$ above $x_0$ corresponds to the smooth model of $E^{\an}_{k'}$. Let $\Gamma=\sk(\mathscr{X})$ and let $\Gamma'=(\pi^{\an})^{-1}\Gamma$. We denote by $y$ and $y'$ the unique nodes of $\Gamma$ and $\Gamma'$ respectively. The situation is depicted in Figure \ref{figure example II} below.

\refstepcounter{equation}
 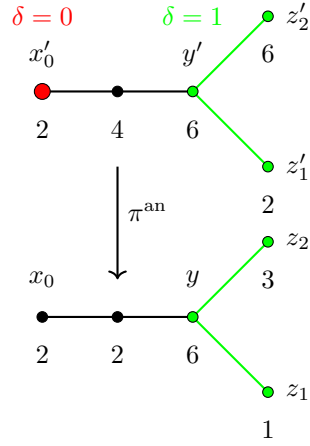
\begin{figure}[ht!]\label{figure example II}
\centering
	\begin{tikzpicture}

\draw[thick,->] (2,2)--(2,0.5);
\node at (2.4,1.4) {$\pi^{\an}$};

\node[green] at (3,4) {$\delta=1$};
\node[red] at (1,4) {$\delta=0$};
\node at (1,3.5) {$x_0'$};
\node at (1,0.5) {$x_0$};
\node at (3,3.5) {$y'$};
\node at (3,0.5) {$y$};

\node at (4.4,4) {$z_2'$};
\node at (4.4,2) {$z_1'$};
\node at (4.4,1) {$z_2$};
\node at (4.4,-1) {$z_1$};

\node at (1,2.5) {$2$};
\node at (2,2.5) {$4$};
\node at (3,2.5) {$6$};
\node at (4,1.5) {$2$};
\node at (4,3.5) {$6$};

\draw[thick] (1,3) -- (3,3);
\draw[green,thick] (4,4) -- (3,3)-- (4,2);

\node at (1,-0.5) {$2$};
\node at (2,-0.5) {$2$};
\node at (3,-0.5) {$6$};
\node at (4,-1.5) {$1$};
\node at (4,0.5) {$3$};
\draw[thick] (1,0) -- (3,0);
\draw[green,thick] (4,1) -- (3,0)-- (4,-1);

\draw[fill=red] (1,3) circle (3pt);
\draw[fill=black] (2,3) circle (2pt);
\draw[fill=green] (3,3) circle (2pt);
\draw[fill=green] (4,4) circle (2pt);
\draw[fill=green] (4,2) circle (2pt);

\draw[fill=black] (1,0) circle (2pt);
\draw[fill=black] (2,0) circle (2pt);
\draw[fill=green] (3,0) circle (2pt);
\draw[fill=green] (4,1) circle (2pt);
\draw[fill=green] (4,-1) circle (2pt);

	\end{tikzpicture}
\caption{\small \emph{Picture of $\Gamma'\to \Gamma.$ The points of $V(\mathscr{X})$ and their preimages along $\pi^{\an}$ are drawn and marked with their multiplicities. The different $\delta=\delta_{\pi^{\an}}$ varies linearly with slope $6$ from {\color{red}$\delta=0$} on $x'_0$} to {\color{green}$\delta=1$} on $y'$.} 
\end{figure}

Let us now check that the multiplicities of the drawn vertices of $\Gamma'$ are as claimed in Figure \ref{figure example II}.
Since $z_1$ and $z_2$ are tame type 2 points of $\Gamma$ we have that $m(z_1')=2$ and $m(z_2')=6$ as multiplicities over $k$, see for instance \cite[2.4]{FT}.
Moreover, we have that $d(x'_0,z_1')=d(x_0,z_1)/2=1/4$ by the balancing condition of \ref{balancing}. Since $m(x'_0)=m(z_1')=2$ and $d(x'_0,z_1')=1/4$ it follows by \ref{specialisation neat} that the interval $[x_0',z_1']$ is \emph{neat}, and so by $d(y',z')=d(y,z)/2=1/12$ and by Lemma \ref{nc interval multiplicities} we have $m(y')=6$.


Note that $\chi(\Gamma')=1-3+3+1=0=\chi(E')$, and so by the criterion of \ref{thm: enlarging skeleta} it follows that $\Gamma'\to\Gamma$ is a simultaneous skeleton, and the normalised base change of $\mathscr{X}$ to $(k')^{\circ}$ is toroidal (in fact \emph{snc}, though we do not need this so we omit the verification).

\paragraph*{Step 3: Study of behaviour of the analytic different.}We can now deduce from Proposition \ref{rh} that the different is linear from $x'_0$ to $y'$. Since $\H(x'_0)/\H(x_0)$ is unramified we have that $\delta_{\H(x'_0)/\H(x_0)}^{\log}=0$. Since $y$ is temperate we have by Proposition \ref{constant along temperate part} that $\delta(y)=\frac{1}{[k':k]}\delta^{\log}_{k'/k}=2/2=1$. We conclude that the different increases with slope $\frac{1}{d(x'_0,y')}=6$ from $x'_0$ to $y'$.

 Alternatively, we could have obtained the previous slope computation by using the Riemann-Hurwitz formula at either $x'_0$ or $y'$. To wit, the Riemann-Hurwitz formula at $x'_0$ is of the form $$4=-\chi(x'_0/f(x'_0))=\frac{\partial_{[x_0',y')}\delta}{2}+2-1,$$ and the RH formula at $y'$ is of the form $$2=-\chi(y'/f(y'))=\left(\frac{\partial_{[y',x_0')}\delta}{6}+2-1\right)+2\left(0+2-1\right).$$
\end{example}

\begin{proof}[Proof of Theorem \ref{thm: ell curve}] Let $E$ be an elliptic curve over $k$ with bad reduction and potentially multiplicative reduction. 
It is well-known that $\nu>0$ if and only if $E$ is potentially multiplicative.
The fact that the monodromy extension is of degree $2$ follows from the theory of Tate uniformisation \cite[V]{silverman advanced topics} as follows. The Tate elliptic curve $E':y^2+xy=x^3+ax+b$ where $(a,b)=(-36/(j(E)-1728),-1/(j(E)-1728))$ is semistable of reduction type $I_{\nu}$ and $j(E)=j(E')$. Since $j\ne 0,12^3$ it follows that $E$ is a quadratic twist of $E'$ and so there exists a degree $2$ extension $k'/k$ such that $(E')_{k'}\cong E_{k}$, and $k'/k$ is the required monodromy extension. The reduction type of $E_{k'}$ is $I_{2\nu}$.

The essential skeleton $\sk(E_{k'})$ coincides with the minimal skeleton of $E_{k'}$ and is homeomorphic to a circle. The fibers of $\pi^{\an}$ are either 2-to-1 (topological splitting) or 1-to-1 (topological ramification). Moreover the topological splitting locus is open since it is the locus where $\pi^{\an}$ is a local isomorphism. Since the image of $\sk(E_{k'})$ in $E^{\an}$ is contractible, it follows that at least one point, and therefore some non-trivial subinterval $I$ of $sk(E_{k'})$ lies in the splitting locus. 

Let $\omega'$ be the image of $\omega=\omega_E$ in $H^0(E_{k'},\omega_{E_{k'}/k})\cong H^0(E,\omega_{E/k})\otimes k'$. Then $\omega'$ is an invariant differential of $E_{k'}$ and by \cite[8.13]{W} we have $$\wt_{\omega_{E'}}(x)=\wt_{\omega}(\pi(x))$$ for all $x\in I$.  Thus $\wt_{\omega}$ is constant on $\pi^{\an}(I)$. By \cite[\S4.4.5]{MN} the weight function $\wt_{\omega}$ strictly increases away from its minimum locus $\sk(E)$. So by inspection of Table \ref{kodaira neron} and the fact that $\pi^{\an}$ has finite fibers it follows that $E^{\an}$ contains an interval on which $\wt_{\omega}$ is constant. Again by inspection of Table \ref{kodaira neron} it follows that $E$ has reduction type $I_n^*$ for some $n\in\Z_{>0}$ because this is the only non-semistable type for which $\sk(E)$ contains an interval.

We now show that $n=4\delta^{\log}_{k'/k}+\nu$. Let $\Gamma=\sk(\mathscr{E}_{\min})$ denote the minimal skeleton and let $\Gamma'\to\Gamma$ denote the normalised base change of $\Gamma$ to $k'/k$. By the above we know that $\Gamma'$ contains the minimal skeleton of $E_{k'}$. In Figure \ref{figure proof pot mult} below we depict $\Gamma'\to\Gamma$.

\refstepcounter{equation}
 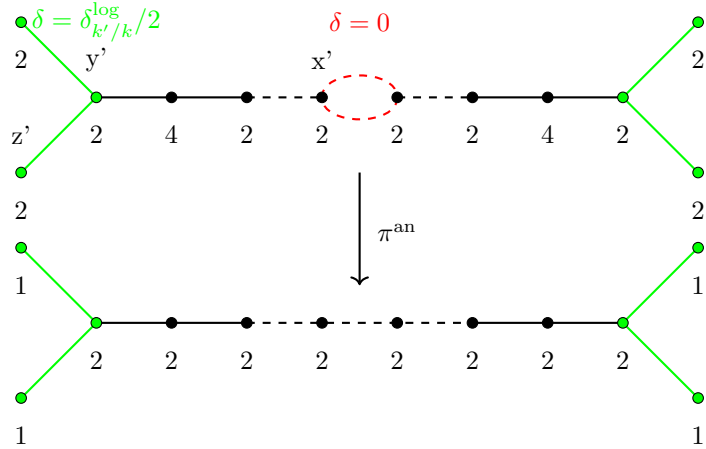
\begin{figure}[ht!]\label{figure proof pot mult}
\centering
	\begin{tikzpicture}

\draw[thick,->] (0.5,2)--(0.5,0.5);
\node at (1,1.2) {$\pi^{\an}$};

\node[green] at (-3,4) {$\delta=\delta^{\log}_{k'/k}/2$};
\node[red] at (0.5,4) {$\delta=0$};

\node at (-4,3.5) {$2$};
\node at (-4,1.5) {$2$};
\node at (-3,2.5) {$2$};
\node at (-2,2.5) {$4$};
\node at (-1,2.5) {$2$};
\node at (-4,2.5) {z'};
\node at (-3,3.5) {y'};
\node at (0,3.5) {x'};
\node at (0,2.5) {$2$};
\node at (1,2.5) {$2$};
\node at (2,2.5) {$2$};
\node at (3,2.5) {$4$};
\node at (4,2.5) {$2$};
\node at (5,1.5) {$2$};
\node at (5,3.5) {$2$};

\draw[thick] (-3,3) -- (-1,3);
\draw[thick,dashed] (-1,3) -- (-0,3);
\draw[thick,dashed] (1,3) -- (2,3);
\draw[thick] (2,3) -- (4,3);
 \draw[red,thick,dashed] (0,3) to [out=90,in=90] (1,3);
 \draw[red,thick,dashed] (1,3) to [in=-90,out=-90] (0,3); 
\draw[green,thick] (5,4) -- (4,3)-- (5,2);
\draw[green, thick] (-4,4) -- (-3,3) --(-4,2);

\node at (-4,0.5) {$1$};
\node at (-4,-1.5) {$1$};
\node at (-3,-0.5) {$2$};
\node at (-2,-0.5) {$2$};

\node at (-1,-0.5) {$2$};
\node at (0,-0.5) {$2$};
\node at (1,-0.5) {$2$};
\node at (2,-0.5) {$2$};
\node at (3,-0.5) {$2$};
\node at (4,-0.5) {$2$};
\node at (5,-1.5) {$1$};
\node at (5,0.5) {$1$};
\draw[thick] (-3,0) -- (-1,0);
\draw[thick] (4,0) -- (2,0);
\draw[thick,dashed] (-1,0) -- (2,0);
\draw[green,thick] (5,1) -- (4,0)-- (5,-1);
\draw[green,thick] (-4,1) -- (-3,0) --(-4,-1);

\draw[fill=green] (-4,4) circle (2pt);
\draw[fill=green] (-4,2) circle (2pt);
\draw[fill=green] (-3,3) circle (2pt);
\draw[fill=black] (-2,3) circle (2pt);
\draw[fill=black] (-1,3) circle (2pt);
\draw[fill=black] (0,3) circle (2pt);
\draw[fill=black] (1,3) circle (2pt);
\draw[fill=black] (2,3) circle (2pt);
\draw[fill=black] (3,3) circle (2pt);
\draw[fill=green] (4,3) circle (2pt);
\draw[fill=green] (5,4) circle (2pt);
\draw[fill=green] (5,2) circle (2pt);

\draw[fill=green] (-4,1) circle (2pt);
\draw[fill=green] (-4,-1) circle (2pt);
\draw[fill=green] (-3,0) circle (2pt);
\draw[fill=black] (-2,0) circle (2pt);
\draw[fill=black] (-1,0) circle (2pt);
\draw[fill=black] (0,0) circle (2pt);
\draw[fill=black] (1,0) circle (2pt);
\draw[fill=black] (2,0) circle (2pt);
\draw[fill=black] (3,0) circle (2pt);
\draw[fill=green] (4,0) circle (2pt);
\draw[fill=green] (5,1) circle (2pt);
\draw[fill=green] (5,-1) circle (2pt);

	\end{tikzpicture}
\caption{\small \emph{Picture of $\Gamma'\to\Gamma$. The markings are the multiplicities of the associated divisorial points. The different $\delta=\delta_f$ varies linearly with slope 2 from {\color{red}$\delta=0$} on split locus} to {\color{green}$\delta=\delta_{k'/k}^{\log}$} above $E^{\mathrm{temp}}$}.
\end{figure}

We claim that $\Gamma'$ is a skeleton. 
In order to verify the criterion of Theorem \ref{thm: enlarging skeleta} it is enough to check that all the leaves and nodes of $\Gamma'$ have multiplicity $1$ (over $k'$). This is proven similarly as in Example \ref{example II} above: the multiplicities of the leaves over $k$ of $\Gamma'$ are $2$ since their projections to $\Gamma$ are tame. Let $y$ be a node of $\Gamma$, and let $y'$ denote its base change. Let $x'$ be the neighbouring node of $\Gamma'$ which also lies on the minimal skeleton (also see the annotations in Figure \ref{figure proof pot mult}). Let $z'$ be a leaf of $\Gamma'$ neighbouring to $y'$. Then $[x',z']$ is neat by \ref{specialisation neat}. By \ref{balancing} it follows that $d(y',z')=d(y,z)/2$, and it follows that $y'$ has multiplicity $2$ by \ref{nc interval multiplicities}.
This means the result of \ref{thm: enlarging skeleta} applies and $\Gamma'\to\Gamma$ is a simultaneous skeleton. 

 A computation with the Riemann-Hurwitz formula (\ref{rh}) at either $x'$ or $y'$, similarly as in Example \ref{example II}, implies that the different increases linearly with slope $2$ from $\delta(x')=0$ to $\delta(y')$. Moreover we have $\delta(y')=\delta^{\log}_{k'/k}/2$ by \ref{constant along temperate part} and so $d(x',y')=\delta^{\log}_{k'/k}/4$. By Proposition \ref{balancing} the metric contracts by a factor of $2$ away from the splitting locus, and $\pi^{\an}$ is an isometry on the splitting locus. In particular $d(x,y)=2d(x',y')=\delta_{k'/k}^{\log}/2$. We now obtain the desired formula via computing the distance between the nodes of $\Gamma$ in two ways:  $$\frac{n}{4}=2d(x,y)+\frac{\nu}{4}=\delta_{k'/k}^{\log}+\frac{\nu}{4}.$$
\end{proof}
\begin{block} In fact, the result of Theorem \ref{thm: ell curve} remains true in the case $\nu=0$, provided one changes ``potentially multiplicative'' to ``potentially \emph{ordinary}'', see Theorem \ref{main2} below. This result is also shown by an explicit computation in \cite[4.2]{L13}.
	\end{block}
	\begin{block}[Slopes of the different]
		To conclude this section, let us remark that in all examples above the different function $\delta$ had 
		either slopes $0,2$ or $6$ on the part where $\delta$ is not trivialised.  Moreover, in the potentially multiplicative and potentially ordinary case only slopes $0,2$ occur. We propose the guiding principle that $\delta$ ``tends'' to behave more complicated, that is with larger slopes, as the potential reduction type becomes ``more supersingular''. Also see \cite[Theorem 7.2.7]{CTT} for a similar phenomenon. 
	\end{block}


 \section{$p$-cyclic arithmetic surface quotient singularities}
\begin{block}[Assumptions]\label{setup obuswewers}
 Let $C$ be as in \ref{notation}. If $g(C)=1$ we assume $C(k)\ne\emptyset$. \emph{In this section, we additionally assume that $C/k$ has \emph{bad} reduction, and obtains \emph{good} reduction over a Galois extension $k'/k$ of degree $p$.} 
\end{block}
\begin{block}[Notation] \label{notation ow} In this section we largely follow the notation of \cite{L14}. We write $\sigma$ for a generator of $\Gal(k'/k)$, and we write $\mathscr{Y}/(k')^{\circ}$ for the smooth model of $C'=C\otimes_k{k'}$. We write $x'$ for the unique vertex of $V(\mathscr{Y})$, this is the divisorial point corresponding to $\mathscr{Y}_s$ and we let $x=\pi^{\an}(x')$ where $\pi^{\an}:C_{k'}^{\an}\to C^{\an}$ denotes the projection.

Note that $\Gal(k'/k)=\langle \sigma \rangle$ acts on both $\mathscr{Y}$ and $\mathscr{Y}_s$. 
Consider the $k^{\circ}$-model $\mathscr{Z}=\mathscr{Y}/\langle \sigma\rangle$ of $C$. The induced map $\mathscr{Y}_s/\langle \sigma\rangle\to\mathscr{Z}_s$ is the normalisation. The ramified $\langle\sigma\rangle$-Galois cover $\mathscr{Y}_s\to\mathscr{Y}_s/\langle \sigma\rangle$ coincides with the map of residual curves $C_{x'}\to C_{x}$. Let $P_1,\dots,P_d\in C_{x'}$ be the ramification points, and denote the images by $Q_1,\dots,Q_d\in C_{x}$. Clearly every singular point of $\mathscr{Z}$ is one of the points $Q_1,\dots,Q_d$, and conversely all these points are singular, by Zariski-Nagata purity of the branch locus (and automatic flatness above the regular locus).
\end{block}
\begin{block}[The $p$-rank of $C_x$]\label{p-rank}
The classical Riemann-Hurwitz formula implies that $$2g(x')-2=p(2g(x)-2)+d(p-1)+\sum_{1\le i\le d}\delta^{\log}_{\O_{P_i}/\O_{Q_i}}$$ and moreover $\delta^{\log}_{\O_{P_i}/\O_{Q_i}}=j_i(p-1)$ where $j_i\ge 1$ denotes the unique jump of the ramification filtration of $\O_{P_i}/\O_{Q_i}$. We call the point $P_i$ \emph{weakly wildly ramified} in case $j_i=1$, that is the jump is as small as possible. 

Let $\gamma(x)$ denote the \emph{$p$-rank of $C_x$}, defined as the $\widetilde{k}$-dimension of $\mathrm{Pic}_{C_x/\widetilde{k}}[p](\widetilde{k})$. Recall that $\gamma(x)\le g(x)$, and we call $C_x$ \emph{ordinary} if equality holds. 
The Deuring-Shafarevich-Crew formula \cite[1.8]{crew} implies that $$2\gamma(C_{x'})-2=p(2\gamma(C_x)-2)+2d(p-1).$$

A computation using the previous formulae shows that $C_x$ is ordinary if and only if $C_{x'}$ is ordinary, and in that case $C_x\to C_{x'}$ is \emph{weakly wildly ramified} at all points, that is $j_1=\dots=j_d=1$. 

For the remainder, we focus on one pair $(P_i,Q_i)=(P,Q)$, and let $\CC$ denote a \emph{snc} resolution of $\mathscr{Z}$. We write $\Gamma=\Gamma_Q$ for the dual graph of $\Gamma$, and we let $\Gamma_Q$ denote the connected component of $\Gamma\setminus \{x\}$ corresponding to $Q$.
\end{block}
\label{sec:quotient}

\begin{proof}[Proof of Theorem \ref{thm: combinatorial slope}] Let $\phi:\Gamma'=(\pi^{\an})^{-1}\Gamma\to \Gamma$ be the base change map. By a theorem of Abhyankar, see \cite[\S12]{K94} and the references cited there, after possibly replacing $\CC$ by a series of blowups in closed points of $\CC$, we may assume that $\Gamma'$ is a skeleton, that is $\Gamma'\to\Gamma$ is a simultaneous skeleton. Note that the value of $\chi(\Gamma_Q)$ does not change after blowing up $\CC$ in a closed point.

Write $b'$ for the branch of $x'$ corresponding to $P$. By Hilbert's formula (as in \ref{example II}) and Proposition \ref{eqn: comparison rh} and it follows that $$(p-1)j_Q=\delta^{\log}_{\O_{P}/\O_{Q}}=\frac{1}{m(x')}\partial_{b'}\delta.$$ Let $\Gamma'_{Q}$ denotes the base change of $\Gamma_{Q}$ to $k'$, then because $\Gamma'_{Q}$ is obtained by a series of blowups as in \ref{blowup sequence}, we obtain by invariance of $\chi(\cdot)$ that $\chi(\Gamma'_{Q})=m(x')=p$. In the Riemann-Hurwitz formula \ref{rh}, we now sum the coefficients of $\Delta(\delta)$ which belong to $\Gamma_Q$. All but one slope come in pairs of cancelling slopes, so we obtain that $$\partial_{b'}\delta=\chi(\Gamma'_{Q})-\deg(\pi)\chi(\Gamma_{Q})=p-p\chi(\Gamma_{Q}).$$ 
\end{proof}
\begin{example}
	In the setting of Example \ref{example II}, we have $p=2$, $d=1$, and with the notation of Figure \ref{figure example II} we compute that $$\chi(\Gamma_{Q})=-m(y)+m(z_1)+m(z_2)=-6+1+3=-2.$$ Note that $C_{x_0'}\to C_{x_0}$ has a unique ramification point $P_1$, corresponding to the branch represented by $[x_0',y_0']$. Theorem \ref{thm: combinatorial slope} then predicts that the unique ramification jump at $P_1$ equals $3$.
\end{example}
\begin{block}[Weakly ramified case]\label{weakly ramified case} Now suppose that $j_Q=1$, i.e. $\O_{P}/\O_{Q}$ is weakly wildly ramified. By Theorem \ref{thm: combinatorial slope} it follows that $\chi(\Gamma_{Q})=2-p$. By \cite[5.3 and 4.3]{L14} the graph $\Gamma_{Q}$ is as depicted in Figure \ref{figure obus wewers} below.
\end{block}
\refstepcounter{equation}
\begin{figure}[ht!]\label{figure obus wewers}
\centering
	\begin{tikzpicture}
		\draw (0,0) circle (2pt);
\draw[fill=black] (1,0) circle (2pt);
\draw[fill=black] (2,0) circle (2pt);
\draw[fill=black] (3,0) circle (2pt);
\draw[fill=black] (4,1) circle (2pt);
\draw[fill=black] (4,-1) circle (2pt);
\draw[fill=black] (5,1) circle (2pt);
\draw[fill=black] (5,-1) circle (2pt);

\node at (0,0.5) {$x$};
\node at (0,-0.5) {$p$};
\node at (1,-0.5) {$p$};
\node at (2,-0.5) {$p$};
\node at (3,-0.5) {$p$};
\node at (3,0.5) {$y$};
\node at (4,-1.5) {$r$};
\node at (5,-1.5) {$1$};
\node at (5.5,-1.25) {$w$};
\node at (4,0.5) {$p-r$};
\node at (5,0.5) {$1$};
\node at (5.5,0.75) {$z$};
\draw[thick] (0,0) -- (1,0);
\draw[thick] (2,0) -- (3,0)--(4,1);
\draw[thick] (3,0) -- (4,-1);
\draw[dashed] (1,0) -- (2,0);
\draw[dashed] (4,-1) -- (5,-1);
\draw[dashed] (4,1) -- (5,1);
	\end{tikzpicture}
\caption{\small \emph{The dual graph $\Gamma_{Q}$ as in \ref{weakly ramified case}. The markings indicate the multiplicities of the corresponding components, and $r>0$ is some integer smaller than $p$. The graph $\Gamma_{Q}$ has a unique node $y$, and the points of $V(\mathscr{X})$ on the segment $[x,y]$ each have multiplicity $p$.}}
\end{figure}
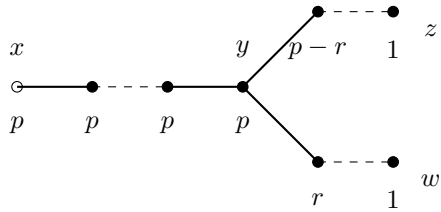

 \begin{block}[Ordinary case]\label{ordinary case}
 	 Suppose that $C$ has ordinary good reduction over $k'$, by \ref{p-rank} this is equivalent to $\mathscr{Y}_s\to \mathscr{Y}_s/\langle \sigma\rangle$ being everywhere weakly wildly ramified. Let $\Gamma=\sk(\mathscr{X})$. It follows that $\Gamma\setminus \{x\}$ is a disjoint union of $d$ copies of the graphs depicted in \ref{figure obus wewers}. Fix one such graph $\Gamma_Q$. 
 \end{block}
\begin{theorem}[Obus-Wewers] \label{main2} Sitation of \ref{ordinary case}. Then the number of vertices of $V(\mathscr{X})$ on the segment $[x,y]$ is equals $jp$, where $j$ is the unique ramification jump of $k'/k$. 
\end{theorem}
\begin{proof} First we verify that $\Gamma'\to\Gamma$ is a simultaneous skeleton of $\pi^{\an}:C_{k'}^{\an}\to C^{\an}$, by verifying the criterion of Theorem \ref{thm: enlarging skeleta}. Let $w$ and $z$ denote the leaves of $\Gamma_{Q}$. Then $w$ and $z$ each have a unique preimage $w'$ and $z'$ in $\Gamma_{P}=\Gamma_{Q}'$ and $m(w')=m(z')=p$. Since $C_{k'}^{\an}$ is contractible the topological ramification locus is connected. It follows that $y$ also has a unique preimage $y'$. The interval $[x',z']$ is neat by \ref{specialisation neat}. By \ref{balancing} it follows that $d(y',z')=d(y,z)/p$, and it follows that $y'$ has multiplicity $p$ by \ref{nc interval multiplicities}. It follows that $\chi(\Gamma')=(\chi(x')-d)+d=\chi(x')=\chi(C')$ and so the criterion of Theorem \ref{thm: enlarging skeleta} is met. Therefore $\Gamma'\to\Gamma$ is a simultaneous skeleton. 

Note that $\delta(x')=0$ since $\H(x')/\H(x)$ is unramified, and $\delta(y')=\frac{1}{p}\delta_{k'/k}^{\log}$ by \ref{constant along temperate part}. By the proof of Theorem \ref{thm: combinatorial slope} it follows that $\partial_{b'}\delta=p(p-1)$, where $b'$ is the germ of the interval $[x',y']$. By Proposition \ref{rh} the different is linear on $[x',y']$ and so it follows $d(x',y')=\delta_{k'/k}^{\log}/(p-1)=j$, where the last equality follows from Hilbert's formula. Hence $d(x,y)=j/p$ by Theorem \ref{balancing}. The distance between any two neighbouring points of $V(\mathscr{X})\cap[x,y]$ is $1/p^2$. Therefore there are indeed $\frac{j/p}{1/p^2}=jp$ vertices on $[x,y]$.
\end{proof}
\begin{remark} Obus-Wewers \cite{OW} actually prove a finer result, namely they only need to assume that $j_Q=1$ (weak wild ramification at $P$) to conclude that the resolution graph has the desired shape. It is plausible that a similar strategy exists to obtain this result, but we do not know how to extend \ref{thm: enlarging skeleta} in this setting.
\end{remark}

\end{document}